\documentclass{amsart}

  \usepackage{verbatim}
\usepackage{mathrsfs}
\usepackage{latexsym}
\usepackage{epsfig}
\usepackage{amsmath}
\usepackage{bbm}
\usepackage{graphics}
\usepackage{amssymb}
\usepackage{amsrefs}
\usepackage{amsthm}
\usepackage{amsopn}
\usepackage{amscd}
\usepackage[all,knot]{xy}
\xyoption{all}
\usepackage{rotating}
\usepackage{hyperref}

\theoremstyle{plain}
\newtheorem{thm}{Theorem}[section]
\newtheorem{lem}[thm]{Lemma}
\newtheorem{prop}[thm]{Proposition}
\newtheorem{cor}[thm]{Corollary}

\newtheorem*{thm*}{Theorem}
\newtheorem*{lem*}{Lemma}
\newtheorem*{prop*}{Proposition}
\newtheorem*{cor*}{Corollary}

\theoremstyle{definition}
\newtheorem{defn}[thm]{Definition}
\newtheorem*{defn*}{Definition}

{}
\newtheorem{rem}[thm]{Remark}
\newtheorem*{rem*}{Remark}
\newtheorem{notation}[thm]{Notation}{}
\newtheorem{convention}[thm]{Convention}{}
{}
\newtheorem*{ack}{Acknowledgements}

\theoremstyle{remark}
{}
{}
{}

\def\cie{\subseteq}

\def\iso{\cong}

\def\Un{\bigcup}

\def\intersec{\cap}

\def\to{\longrightarrow}

\def\nat{\mathbb{N}}

\def\int{\mathbb{Z}}
\def\real{\mathbb{R}}

\def\str{\mathcal{O}}

\def\a{\alpha}
\def\b{\beta}

\def\l{\lambda}
\def\r{\rho}
\def\s{\sigma}
\def\S{\Sigma}
\def\t{\tau}

\def\module{\mathrm{mod}}

\def\Module{\mathrm{Mod}}

\DeclareMathOperator{\depth}{depth}

\DeclareMathOperator{\Spec}{Spec}
\DeclareMathOperator{\Sing}{Sing}

\DeclareMathOperator{\Spc}{Spc}

\DeclareMathOperator{\supp}{supp}

\DeclareMathOperator{\codim}{codim}

\DeclareMathOperator{\rk}{rk}
\DeclareMathOperator{\id}{id}

\DeclareMathOperator{\Hom}{Hom}
\DeclareMathOperator{\Ext}{Ext}
\DeclareMathOperator{\Tor}{Tor}

\DeclareMathOperator{\hocolim}{hocolim}
\DeclareMathOperator{\Coh}{Coh}
\DeclareMathOperator{\QCoh}{QCoh}

\DeclareMathOperator{\Flat}{Flat}

\DeclareMathOperator{\GInj}{GInj}
\DeclareMathOperator{\Inj}{Inj}

\DeclareMathOperator{\pdim}{pd}
\DeclareMathOperator{\idim}{id}

\DeclareMathOperator{\gldim}{gl. dim}

\DeclareMathOperator{\cx}{cx}

\title[Subcategories of singularity categories]{Subcategories of singularity categories via tensor actions}
\author{Greg Stevenson}
\address{Universit\"at Bielefeld, Fakult\"at f\"ur Mathematik, BIREP Gruppe, Postfach 10\,01\,31, 33501 Bielefeld, Germany.}
\email{gstevens@math.uni-bielefeld.de}

\begin{document}

\subjclass[2000]{ 
14M10, 18E30  
(14F05, 13C40)} 

\keywords{Singularity category, complete intersection, localizing subcategory, triangular geometry, telescope conjecture}

\begin{abstract}
\noindent We obtain, via the formalism of tensor actions, a complete classification of the localizing subcategories of the stable derived category of any affine scheme that has hypersurface singularities or is a complete intersection in a regular scheme; in particular this classifies the thick subcategories of the singularity categories of such rings. The analogous result is also proved for certain locally complete intersection schemes.  It is also shown that from each of these classifications one can deduce the (relative) telescope conjecture.
\end{abstract}

\maketitle

\tableofcontents

\section{Introduction}
Since the pioneering work of Devinatz, Hopkins, and Smith \cite{DevinatzHopkinsSmith} and Neeman \cite{NeeChro} there has been sustained interest in classifying localizing and thick subcategories of triangulated categories. Since these results several new examples have been understood, for example the work of Benson, Carlson, and Rickard \cite{BCR} on modular representations, Thomason \cite{Thomclass} on perfect complexes over quasi-compact quasi-separated schemes, Friedlander and Pevtsova \cite{FriedlanderPevtsova} on finite group schemes, and many others. Recently there has also been very interesting work, namely the tensor triangular geometry of Balmer \cite{BaSpec}, \cite{BaSSS}, and Balmer-Favi \cite{BaRickard}, and the stratifications of Benson, Iyengar, and Krause \cite{BIK}, \cite{BIKstrat}, \cite{BIKstrat2}, providing conceptual frameworks uniting these classification theorems.

The purpose of this paper is two-fold: to dissect another class of triangulated categories and lay bare the structure of their lattice of localizing subcategories and to provide the promised application of the formalism developed in \cite{StevensonActions}. Let us be explicit that this really is an application of ``relative tensor triangular geometry'' and not a restatement of results that one knows can be proved with the current technology; there is no obvious, at least to the author, tensor structure on the categories we consider and some of our results have a non-affine character that, at least at the current time, seems to place them out of reach of the theory of stratifications.

We now give some details about our results. Suppose $X$ is a noetherian separated scheme. Then one defines a category
\begin{displaymath}
D_{\mathrm{Sg}}(X) := D^b(\Coh X) / D^{\mathrm{perf}}(X),
\end{displaymath}
where $D^b(\Coh X)$ is the bounded derived category of coherent sheaves on $X$ and $D^\mathrm{perf}(X)$ is the full subcategory of complexes locally isomorphic to bounded complexes of finitely generated projectives. This category measures the singularities of $X$. In particular, $D_{\mathrm{Sg}}(X)$ vanishes if and only if $X$ is regular, it is related to other measures of the singularities of $X$ for example maximal Cohen-Macaulay modules (see \cite{Buchweitzunpub}), and its properties reflect the severity of the singularities of $X$. The particular category which will concern us is the stable derived category of Krause \cite{KrStab}, namely
\begin{displaymath}
S(X) := K_\mathrm{ac}(\Inj X)
\end{displaymath}
the homotopy category of acyclic complexes of injective quasi-coherent $\str_X$-modules. We slightly abuse standard terminology by calling $S(X)$ the \emph{singularity category} of $X$. The singularity category is a compactly generated triangulated category whose compact objects are equivalent to $D_{\mathrm{Sg}}(X)$ up to summands.

We show that the unbounded derived category of quasi-coherent sheaves of $\str_X$-modules, which we denote $D(X)$, acts on the singularity category $S(X)$. As in \cite{StevensonActions} this gives rise to a theory of supports for objects of $S(X)$, and hence $D_{\mathrm{Sg}}(X)$, taking values in $X$. In particular the formalism gives rise to assignments
\begin{displaymath}
\s(\mathcal{L}) = \supp \mathcal{L} = \{x \in X \; \vert \; \mathit{\Gamma}_x \mathcal{L} \neq 0\}
\end{displaymath}
and 
\begin{displaymath}
\t(W) = \{A \in S(X) \; \vert \; \supp A \cie W\}
\end{displaymath}
for an action closed localizing subcategory $\mathcal{L}\cie S(X)$ and a subset $W\cie X$. We note that, as one would expect, the support actually takes values in $\Sing X$ the singular locus of $X$.

The first main result extends work of Takahashi \cite{TakahashiMCM} who has classified thick subcategories of $D_{\mathrm{Sg}}(R)$ when $R$ is a local hypersurface. We are able to extend this result to cover all localizing subcategories of $S(R)$ as well as removing the hypothesis that $R$ be local by using the action of $D(R)$:

\begin{thm*}[\ref{cor_hyper_min}]
If $R$ is a noetherian ring which is locally a hypersurface then there is an order preserving bijection
\begin{displaymath}
\left\{ \begin{array}{c}
\text{subsets of}\; \Sing R 
\end{array} \right\}
\xymatrix{ \ar[r]<1ex>^{\tau} \ar@{<-}[r]<-1ex>_{\sigma} &} \left\{
\begin{array}{c}
\text{localizing subcategories of} \; S(R) \\
\end{array} \right\}. 
\end{displaymath}
It follows that there are also order preserving bijections
\begin{displaymath}
\left\{ \begin{array}{c}
\text{specialization closed} \\ \text{subsets of}\; \Sing R 
\end{array} \right\}
\xymatrix{ \ar[r]<1ex>^{\tau} \ar@{<-}[r]<-1ex>_{\sigma} &} \left\{
\begin{array}{c}
\text{localizing subcategories of} \; S(R)\; \\
\text{generated by objects of} \; S(R)^c
\end{array} \right\} 
\end{displaymath}
and
\begin{displaymath}
\left\{ \begin{array}{c}
\text{specialization closed} \\ \text{subsets of}\; \Sing R 
\end{array} \right\}
\xymatrix{ \ar[r]<1ex> \ar@{<-}[r]<-1ex> &} \left\{
\begin{array}{c}
\text{thick subcategories of} \; D_{\mathrm{Sg}}(R) \\
\end{array} \right\}.
\end{displaymath}
\end{thm*}
From this result we are able to deduce the telescope conjecture for $S(R)$ when $R$ is locally a hypersurface.

We next approach the problem of understanding the structure of $S(X)$ where $X$ is a noetherian separated scheme with hypersurface singularities. Let us say that a localizing subcategory of $S(X)$ is a \emph{localizing submodule} if it is closed under the $D(X)$ action. We are able to globalise our results on affine hypersurfaces to give a classification of localizing submodules of $S(X)$. We summarise our main results on non-affine hypersurfaces in the following theorem which is a combination of the results of \ref{thm_hyper_nonaffine_bijections}, \ref{cor_ample_stable}, and \ref{thm_nonaffine_hyper_tele}.

\begin{thm*}
Suppose $X$ is a noetherian separated scheme with only hypersurface singularities. Then there is an order preserving bijection
\begin{displaymath}
\left\{ \begin{array}{c}
\text{subsets of}\; \Sing X 
\end{array} \right\}
\xymatrix{ \ar[r]<1ex>^{\tau} \ar@{<-}[r]<-1ex>_{\sigma} &} \left\{
\begin{array}{c}
\text{localizing submodules of} \; S(X) \\
\end{array} \right\} .
\end{displaymath}
This restricts to the equivalent  bijections
\begin{displaymath}
\left\{ \begin{array}{c}
\text{specialization closed} \\ \text{subsets of}\; \Sing X 
\end{array} \right\}
\xymatrix{ \ar[r]<1ex>^{\tau} \ar@{<-}[r]<-1ex>_{\sigma} &} \left\{
\begin{array}{c}
\text{submodules of} \; S(X)\; \text{generated} \\
\text{by objects of} \; S(X)^c
\end{array} \right\} 
\end{displaymath}
and
\begin{displaymath}
\left\{ \begin{array}{c}
\text{specialization closed} \\ \text{subsets of}\; \Sing X 
\end{array} \right\}
\xymatrix{ \ar[r]<1ex> \ar@{<-}[r]<-1ex> &} \left\{
\begin{array}{c}
\text{thick} \; D^{\mathrm{perf}}(X)\text{-submodules of} \; D_{\mathrm{Sg}}(X) \\
\end{array} \right\}.
\end{displaymath}
Furthermore, if $X$ is the zero scheme of a regular section of an ample line bundle on a regular scheme then every localizing subcategory is a localizing submodule and this gives a complete description of the lattice of localizing subcategories. In either case the relative telescope conjecture is satisfied.
\end{thm*}

Our last main result concerns schemes which are locally complete intersections. By a theorem of Orlov \cite{OrlovSing2}, which has been extended by Burke and Walker \cite{BurkeMF2}, if $X$ is a noetherian separated local complete intersection scheme, admitting a suitable embedding into a regular scheme, the category $D_{\mathrm{Sg}}(X)$ is equivalent to $D_{\mathrm{Sg}}(Y)$ for a hypersurface $Y$ which is given explicitly. We prove that this equivalence extends to an equivalence of $S(X)$ and $S(Y)$. We are thus able to reduce the classification problem for such local complete intersection schemes to the corresponding classification problem for hypersurfaces, which we have already solved (for submodules at least). The classification this gives rise to can be found in Theorem \ref{thm_general_bijections}. As a special case we are able to completely classify the localizing subcategories of $S(R)$ when $R$ is a local (non-abstract) complete intersection ring.

\begin{thm*}[\ref{cor_ci_win}]
Suppose $(R,\mathfrak{m},k)$ is a local complete intersection. Then there is an order preserving bijection
\begin{displaymath}
\left\{ \begin{array}{c}
\text{subsets of}\; \\ \coprod_{\mathfrak{p}\in \Sing R}\limits \mathbb{P}^{c_\mathfrak{p} -1}_{k(\mathfrak{p})}
\end{array} \right\}
\xymatrix{ \ar[r]<1ex>^{\tau} \ar@{<-}[r]<-1ex>_{\sigma} &} \left\{
\begin{array}{c}
\text{localizing subcategories of} \; S(R) \\
\end{array} \right\} 
\end{displaymath}
where $c_\mathfrak{p}$ is the codimension of the singularity at the closed point of $R_\mathfrak{p}$. Furthermore, the telescope conjecture holds for $S(R)$.
\end{thm*}
This gives a different, more general, proof of a similar result for local complete intersection rings announced by Iyengar \cite{IyengarLCI}. We note that there is also a classification for thick subcategories of $D_{\mathrm{Sg}}(R)$ and hence for the stable category of maximal Cohen-Macaulay modules over $R$. However, some more notation is required to state this classification and so we do not give it here (it can be found in Corollary \ref{cor_ci_win}).

We now sketch the layout of the paper. Section \ref{sec_prelims_actions} contains brief preliminaries on tensor actions. In Section \ref{sec_singcat} we introduce the categories we will study and then recall in Section \ref{sec_prelims_alg} a few facts and definitions from commutative algebra. Then in Section \ref{sec_affine} we begin to study these categories in the affine case; the main result here is a technical one identifying the image of the assignment $\t$ from which our classification results will follow. Section \ref{sec_hypersurface} is devoted to proving the classification theorem for affine schemes with hypersurface singularities. We show in Section \ref{sec_schemes} how to globalise our affine results  and give the classification for schemes with hypersurface singularities. The last three sections contain the generalization, via a theorem of Orlov, to local complete intersection schemes, the proof that the choices made in applying Orlov's theorem do not matter, and an explicit treatment of the case of local complete intersection rings.

\begin{ack}
This paper contains material from my PhD thesis written at the Australian National University under the supervision of Amnon Neeman. I would like to thank Amnon for his support and innumerable valuable comments on many earlier versions of this work. I would like to thank the anonymous examiners of my thesis for their helpful comments. Thanks are also due to Paul Balmer, Jesse Burke, and Srikanth Iyengar for comments on a preliminary version of this manuscript.
\end{ack}

\section{Preliminaries on tensor actions}\label{sec_prelims_actions}
Our purpose here is to highlight some of the main results of \cite{StevensonActions} which will be required to prove our classification results. Of course our summary is not exhaustive, nor even sufficient to cover all of the technical parts of the formalism that will be required. We thus adopt the following convention:

\begin{convention}
To avoid endlessly repeating phrases such as ``Lemma x.y of \cite{StevensonActions}'' we adopt the convention that ``Lemma A.x.y'' refers to Lemma x.y of \cite{StevensonActions}.
\end{convention}

To begin let us take the opportunity to fix some notation and terminology which will be used throughout the paper.

\begin{notation}\label{notation1}
Let $X$ be a topological space (we note that all our spaces will be spectral in the sense of Hochster \cite{Hochsterspectral}). For $x\in X$ we set
\begin{displaymath}
\mathcal{V}(x) = \overline{\{x\}}.
\end{displaymath}
We will say $x$ \emph{specializes} to $y$ if $y\in \mathcal{V}(x)$ and we say a subset $\mathcal{V}\cie X$ is \emph{specialization closed} if $x\in \mathcal{V}$ implies $\mathcal{V}(x)\cie \mathcal{V}$. We denote by
\begin{displaymath}
\mathcal{Z}(x) = \{y\in X \; \vert \; x \notin \mathcal{V}(y)\}
\end{displaymath}
the set of points not specializing to $x$ and denote its complement by
\begin{displaymath}
\mathcal{U}(x) = \{y\in X \; \vert \; x \in \mathcal{V}(y)\}.
\end{displaymath}
\end{notation}

\begin{notation}
Let $\mathcal{T}$ be a triangulated category with all small coproducts. We will denote by $\mathcal{T}^c$ the subcategory of compact objects of $\mathcal{T}$ and we recall that this is a thick triangulated subcategory of $\mathcal{T}$.
\end{notation}

Let us now define the class of triangulated categories whose actions we will consider.

\begin{defn}
Throughout by a \emph{rigidly-compactly generated tensor triangulated category} $(\mathcal{T},\otimes,\mathbf{1})$ we mean a compactly generated triangulated category $\mathcal{T}$ together with a symmetric monoidal structure such that the monoidal product $\otimes$ is a coproduct preserving exact functor in each variable. We also require that the compact objects form a rigid tensor triangulated subcategory $(\mathcal{T}^c,\otimes,\mathbf{1})$. We recall that $\mathcal{T}^c$ is rigid if for all $x$ and $y$ in $\mathcal{T}^c$, setting $x^{\vee} = \hom(x,\mathbf{1})$, the natural map
\begin{displaymath}
x^{\vee}\otimes y \to \hom(x,y)
\end{displaymath}
is an isomorphism, where $\hom(-,-)$ denotes the internal hom which is guaranteed to exist in this case by Brown representability
\end{defn}

We will assume that $\mathcal{T}$ is a rigidly-compactly tensor triangulated category and that $\Spc \mathcal{T}^c$, as defined in \cite{BaSpec}, is a noetherian topological space throughout this section (the reader should not fear if she is unfamiliar with this spectrum, in all our applications it will be explicitly described as the space underlying some given noetherian scheme). Suppose $\mathcal{K}$ is a compactly generated triangulated category. Let us indicate what it means for $\mathcal{T}$ to act on $\mathcal{K}$.

\begin{defn}
A \emph{left action} of $\mathcal{T}$ on $\mathcal{K}$ is a functor 
\begin{displaymath}
*\colon \mathcal{T}\times \mathcal{K} \to \mathcal{K}
\end{displaymath}
which is exact and coproduct preserving in each variable together with natural isomorphisms
\begin{displaymath}
a\colon *\;(\otimes \times \id_\mathcal{K}) \stackrel{\sim}{\to} *\;(\id_\mathcal{T}\times *)
\end{displaymath}
and
\begin{displaymath}
l\colon \mathbf{1}* \stackrel{\sim}{\to} \id_\mathcal{K},
\end{displaymath}
compatible with the biexactness of $(-)*(-)$, such that the diagrams one would expect to commute are commutative; for full details see Definition A.3.2. We only consider left actions and so we shall neglect to include the left from this point forward.

We say that a localizing subcategory $\mathcal{L}$ of $\mathcal{K}$ is a localizing $\mathcal{T}$\emph{-submodule} if it is closed under the action of $\mathcal{T}$.
\end{defn}

Now let us recall some of the theory of \cite{BaRickard} that we will use. Given a specialization closed subset $\mathcal{V}\cie \Spc \mathcal{T}^c$ we denote by $\mathcal{T}^c_\mathcal{V}$ the thick subcategory of compact objects supported (in the sense of \cite{BaSpec}) on $\mathcal{V}$. We let $\mathcal{T}_\mathcal{V}$ be the localizing subcategory generated by $\mathcal{T}^c_\mathcal{V}$ and note that $\mathcal{T}_\mathcal{V}$ is smashing as it is generated by compact objects of $\mathcal{T}$. Let us spend a little time spelling out the consequences of this fact. The subcategory $\mathcal{T}_\mathcal{V}$ gives rise to a smashing localization sequence
\begin{displaymath}
\xymatrix{
\mathit{\mathit{\Gamma}}_\mathcal{V}\mathcal{T} \ar[r]<0.5ex>^(0.6){i_*} \ar@{<-}[r]<-0.5ex>_(0.6){i^!} & \mathcal{T} \ar[r]<0.5ex>^(0.5){j^*} \ar@{<-}[r]<-0.5ex>_(0.5){j_*} & L_\mathcal{V}\mathcal{T}
}
\end{displaymath}
i.e., all four functors are exact and coproduct preserving, $i_*$ and $j_*$ are fully faithful, $i^!$ is right adjoint to $i_*$, and $j_*$ is right adjoint to $j^*$. In particular there are associated coproduct preserving acyclization and localization functors given by $i_*i^!$ and $j_*j^*$ respectively. As in \cite{HPS} Definition 3.3.2 this gives rise to Rickard idempotents which we denote by $\mathit{\mathit{\Gamma}}_\mathcal{V}\mathbf{1}$ and $L_\mathcal{V}\mathbf{1}$ with the property that 
\begin{displaymath}
i_*i^! \iso \mathit{\mathit{\Gamma}}_\mathcal{V}\mathbf{1}\otimes(-) \quad \text{and} \quad j_*j^* \iso L_\mathcal{V}\mathbf{1}\otimes(-).
\end{displaymath}
It follows that they are $\otimes$-orthogonal by the usual properties of localization and acyclization functors. We will also sometimes write $\mathit{\mathit{\Gamma}}_\mathcal{V}\mathcal{T}$ for the category associated to $\mathcal{V}$.

Let us now suppose that we have an action of $\mathcal{T}$ on $\mathcal{K}$. The localization sequences we have just recalled give rise to a theory of supports on $\mathcal{K}$ (mimicking \cite{BaRickard} where the case of $\mathcal{T}$ acting on itself is considered). For each $x\in \Spc\mathcal{T}^c$ there is a $\otimes$-idempotent object $\mathit{\Gamma}_x\mathbf{1}$ defined by 
\begin{displaymath}
\mathit{\Gamma}_x\mathbf{1} = \mathit{\Gamma}_{\mathcal{V}(x)}\mathbf{1} \otimes L_{\mathcal{Z}(x)}\mathbf{1}.
\end{displaymath}
For an object $A$ of $\mathcal{K}$ we set
\begin{displaymath}
\supp A = \{x\in \Spc\mathcal{T}^c \; \vert \; \mathit{\Gamma}_x\mathbf{1}*A \neq 0\}.
\end{displaymath}
Usually, as the action is clear, we supress it and simply write $\mathit{\Gamma}_xA$. One then hopes that, by judiciously choosing the correct action, this support gives a classification of the localizing subcategories of $\mathcal{K}$. The following definition, inspired by \cite{BIKstrat2}, describes the situation one wishes to begin with to prove such classification theorems.

\begin{defn}\label{defn_ltg}
We say $\mathcal{T}\times\mathcal{K} \stackrel{*}{\to} \mathcal{K}$ satisfies the \emph{local-to-global principle} if for each $A$ in $\mathcal{K}$
\begin{displaymath}
\langle A \rangle_* = \langle \mathit{\Gamma}_x A \; \vert \; x\in \Spc \mathcal{T}^c\rangle_*
\end{displaymath}
where $\langle A \rangle_*$ and $\langle \mathit{\Gamma}_x A \; \vert \; x\in \Spc \mathcal{T}^c\rangle_*$ are the smallest localizing subcategories of $\mathcal{K}$ containing $A$ or the $\mathit{\Gamma}_x A$ respectively and closed under the action of $\mathcal{T}$.
\end{defn}
Fortunately this seemingly strong hypothesis on the action turns out only to depend on $\mathcal{T}$ and to hold rather generally.
\begin{thm}[Theorem A.6.9]\label{thm_general_ltg}
Suppose $\mathcal{T}$ is a rigidly-compactly generated tensor triangulated category with a monoidal model and that $\Spc \mathcal{T}^c$ is noetherian. Then the following statements hold:
\begin{itemize}
\item[$(i)$] The local-to-global principle holds for the action of $\mathcal{T}$ on itself;
\item[$(ii)$] The associated support function detects vanishing of objects i.e., $A \in \mathcal{T}$ is zero if and only if $\supp A = \varnothing$;
\item[$(iii)$] For any chain $\{\mathcal{V}_i\}_{i\in I}$ of specialization closed subsets of $\Spc \mathcal{T}^c$ with union $\mathcal{V}$ there is an isomorphism
\begin{displaymath}
\mathit{\Gamma}_\mathcal{V}\mathbf{1} \iso \hocolim \mathit{\Gamma}_{\mathcal{V}_i}\mathbf{1}
\end{displaymath}
where the structure maps are the canonical ones.
\end{itemize}
Furthermore, the relative versions of (i) and (ii) hold for any action of $\mathcal{T}$ on a compactly generated triangulated category $\mathcal{K}$.
\end{thm}

While trying to prove the classification theorem one may as well be greedy and ask for even more; the following definition formalises one of the obvious items on the wish list.

\begin{defn}\label{defn_relative_tele}
We say the \emph{relative telescope conjecture} holds for $\mathcal{K}$ with respect to the action of $\mathcal{T}$ if every smashing $\mathcal{T}$-submodule $\mathcal{S}\cie \mathcal{K}$ (this means $\mathcal{S}$ is a localizing submodule with an associated coproduct preserving localization functor) is generated, as a localizing subcategory, by compact objects of $\mathcal{K}$. 
\end{defn}
We give below a sufficient condition for the relative telescope conjecture to hold for the action of $\mathcal{T}$ on $\mathcal{K}$. In order to state this result we first need to introduce assignments relating subsets of $\Spc \mathcal{T}^c$ and localizing submodules of $\mathcal{K}$. To have to introduce these assignments is certainly no great burden as they are precisely the maps we would like to prove give a classification of localizing subcategories in our application.
\begin{defn}\label{defn_vis_sigmatau}
There are order preserving assignments
\begin{displaymath}
\left\{ \begin{array}{c}
\text{subsets of}\; \Spc\mathcal{T}^c 
\end{array} \right\}
\xymatrix{ \ar[r]<1ex>^\t \ar@{<-}[r]<-1ex>_\s &} \left\{
\begin{array}{c}
\text{localizing submodules} \; \text{of} \; \mathcal{K} \\
\end{array} \right\} 
\end{displaymath}
where for a localizing submodule $\mathcal{L}$ we set
\begin{displaymath}
\s(\mathcal{L}) = \supp \mathcal{L} = \{x \in \Spc\mathcal{T}^c \; \vert \; \mathit{\Gamma}_x\mathbf{1}*\mathcal{L} \neq 0\}
\end{displaymath}
and for a subset $W$ of $\Spc \mathcal{T}^c$ we let
\begin{displaymath}
\t(W) = \{A \in \mathcal{K} \; \vert \; \supp A \cie W\}.
\end{displaymath}
We take both the subsets and subcategories to be ordered by inclusion.
\end{defn}
Our theorem is:
\begin{thm}[Theorem A.7.15]\label{thm_rel_tele}
Suppose $\mathcal{T}$ is rigidly-compactly generated and has a monoidal model. Let $\mathcal{T}$ act on a compactly generated triangulated category $\mathcal{K}$ so that the support of any compact object of $\mathcal{K}$ is a specialization closed subset of $\s\mathcal{K}$ and for each irreducible closed subset $\mathcal{V}$ in $\s\mathcal{K}$ there exists a compact object whose support is precisely $\mathcal{V}$. Furthermore, suppose the assignments $\s$ and $\t$ give a bijection between localizing submodules of $\mathcal{K}$ and subsets of $\s\mathcal{K}$. Then the relative telescope conjecture holds for $\mathcal{K}$ i.e., every smashing $\mathcal{T}$-submodule of $\mathcal{K}$ is generated by objects compact in $\mathcal{K}$.
\end{thm}

\section{The singularity category of a scheme}\label{sec_singcat}

We begin by introducing, for a noetherian separated scheme $X$, an infinite completion $S(X)$, as in Krause's \cite{KrStab}, of the usual singularity category $D_{\mathrm{Sg}}(X)$ (\cite{Buchweitzunpub}, \cite{OrlovSing}). This completion has a natural action of the unbounded derived category of $X$. We can thus bring the machinery of \cite{StevensonActions} to bear on the problem of determining the structure of the lattice of localizing subcategories of $S(X)$.

Throughout we will denote by $X$ a separated noetherian scheme. We will use the following notation
\begin{displaymath}
D(X) := D(\QCoh X), \quad \text{and} \quad K(X) := K(\QCoh X)
\end{displaymath}
where $\QCoh X$ is the category of quasi-coherent sheaves of $\str_X$-modules. We write $D^{\mathrm{perf}}(X)$ for the thick subcategory of perfect complexes and recall that this is precisely the subcategory of compact objects of $D(X)$. We denote by $\Inj X$ the category of injective quasi-coherent sheaves of $\str_X$-modules, the category of coherent $\str_X$-modules by $\Coh X$ and the category of flat $\str_X$-modules by $\Flat X$.

\begin{rem*}
We have defined $\Inj X$ to be the category of injective objects in $\QCoh X$, but we could just as well have taken it to be the category of those injective objects in the category of all $\str_X$-modules which are quasi-coherent. This fact can be found as Lemma 2.1.3 of \cite{ConradDuality}. We thus feel free to speak either of quasi-coherent injective $\str_X$-modules or injective quasi-coherent $\str_X$-modules as they are the same thing when $X$ is (locally) noetherian which is the only case we consider.
\end{rem*}

Given a noetherian separated scheme $X$ we set
\begin{displaymath}
D_{\mathrm{Sg}}(X) = D^b(\Coh X)/ D^{\mathrm{perf}}(X).
\end{displaymath}
This category, usually called the singularity category, provides a measure of the singularities of the scheme $X$. Throughout we will prefer to work with an infinite completion of $D_{\mathrm{Sg}}(X)$ and will reserve the term singularity category for this larger category:

\begin{thm}[\cite{KrStab} Theorem 1.1]\label{thm_recol}
Let $X$ be a noetherian separated scheme.
\begin{itemize}
\item[$(1)$] There is a recollement 
\begin{displaymath}
\xymatrix{
S(\QCoh X) \ar[rr]^{I} && K(\Inj X) \ar[rr]^{Q} \ar@/^1pc/[ll]^{I_\r} \ar@/_1pc/[ll]_{I_{\l}} &&  D(X) \ar@/^1pc/[ll]^{Q_\r} \ar@/_1pc/[ll]_{Q_{\l}}
}
\end{displaymath}
where each functor is right adjoint to the one above it. We call $S(\QCoh X) = K_{\mathrm{ac}}(\Inj X)$, the homotopy category of acyclic complexes of injective quasi-coherent $\str_X$-modules, the \emph{singularity category} of $X$.
\item[$(2)$] The triangulated category $K(\Inj X)$ is compactly generated, and $Q$ induces an equivalence
\begin{displaymath}
K(\Inj X)^{\mathrm{c}} \to D^b(\Coh X).
\end{displaymath}
\item[$(3)$] The sequence
\begin{displaymath}
\xymatrix{
D(X) \ar[r]^(0.5){Q_\l} & K(\Inj X) \ar[r]^(0.5){I_{\l}} & S(\QCoh X)
}
\end{displaymath}
is a localization sequence. Therefore $S(\QCoh X)$ is compactly generated, and $I_\l \circ Q_\r$ induces (up to direct factors) an equivalence
\begin{displaymath}
D_{\mathrm{Sg}}(X) \to S(\QCoh X)^{\mathrm{c}}.
\end{displaymath}
\end{itemize}
\end{thm}

\begin{notation}
As in the theorem we call $S(\QCoh X)$ the \emph{singularity category} of $X$ and we shall denote it by $S(X)$. By (3) of the theorem $S(X)$ essentially contains $D_{\mathrm{Sg}}(X)$; the closure under summands of the image of $D_{\mathrm{Sg}}(X)$ in $S(X)$ is $S(X)^c$, the thick subcategory of compact objects, so it is reasonable to call $S(X)$ the singularity category.
\end{notation}

Now let us prove there is an action 
\begin{displaymath}
D(X) \times S(X) \stackrel{\odot}{\to} S(X)
\end{displaymath}
in the sense of \cite{StevensonActions}.

Recall from above that $S(X)$ is a compactly generated triangulated category. Consider $E = \coprod_\l E_\l$ where $E_\l$ runs through a set of representatives for the isomorphism classes of compact objects in $S(X)$. We define a homological functor $H\colon K(\Flat X) \to \mathrm{Ab}$ by setting, for $F$ an object of $K(\Flat X)$,
\begin{displaymath}
H(F) = H^0(F\otimes_{\str_X} E)
\end{displaymath}
where the tensor product is taken in $K(X)$. This is a coproduct preserving homological functor since we are merely composing the exact coproduct preserving functor $(-)\otimes_{\str_X} E$ with the coproduct preserving homological functor $H^0$.

We remind the reader of the notion of pure acyclicity. In \cite{MurfetTAC} a complex $F$ in $K(\Flat X)$ is defined to be \emph{pure acyclic} if it is exact and has flat syzygies. Such complexes form a triangulated subcategory of $K(\Flat X)$ which we denote by $K_{\mathrm{pac}}(\Flat X)$ and we say that a morphism with pure acyclic mapping cone is a \emph{pure quasi-isomorphism}. We recall that when $X$ is noetherian the tensor product of a complex of flats with a complex of injectives is again a complex of injectives. Note that tensoring a pure acyclic complex of flats with a complex of injectives yields a contractible complex. This can be checked locally using \cite{NeeFlat} Corollary 9.7, see for example \cite{MurfetThesis} Lemma 8.2. In particular every pure acyclic complex lies in the kernel of $H$.

\begin{defn}
With notation as above we denote by $A_{\otimes}(\Inj X)$ the quotient \\ $\ker(H)/K_{\mathrm{pac}}(\Flat X)$, where
\begin{displaymath}
\ker(H) = \{F\in K(\Flat X) \; \vert \; H(\Sigma^i F) = 0 \; \forall i \in \int\},
\end{displaymath}
and by $N(\Flat X)$ the quotient $K(\Flat X)/K_{\mathrm{pac}}(\Flat X)$.
\end{defn}

\begin{lem}
An object $F$ of $K(\Flat X)$ lies in $\ker(H)$ if and only if the exact functor
\begin{displaymath}
F\otimes_{\str_X} (-) \colon K(\Inj X) \to K(\Inj X)
\end{displaymath}
restricts to
\begin{displaymath}
F\otimes_{\str_X} (-) \colon S(X) \to S(X).
\end{displaymath}
In particular, $A_\otimes(\Inj X)$ consists of the pure quasi-isomorphism classes of objects which act on $S(X)$.
\end{lem}
\begin{proof}
Since $F$ is in $\ker(H)$ if and only if $F\otimes_{\str_X} E$ is acyclic it is sufficient to show that $F\otimes_{\str_X} E$ is acyclic if and only if $F\otimes_{\str_X}(-)$ preserves acyclity of complexes of injectives. The if part of this statement is trivial.

On the other hand if $F\otimes_{\str_X} E$ is acyclic then it follows that $F\otimes_{\str_X}(-)$ preserves acyclicity of complexes in the localizing subcategory $\langle E \rangle_{\mathrm{loc}}$ of $K(\Inj X)$. But this is precisely $S(X)$; clearly $\langle E \rangle_{\mathrm{loc}}$ contains a compact generating set for $S(X)$ as it is localizing and hence closed under splitting idempotents so contains all compact objects of $S(X)$.
\end{proof}

\begin{defn}
We say that a complex of flat $\str_X$-modules $F$ is \emph{K-flat} provided $F \otimes_{\str_X}(-)$ sends quasi-isomorphisms to quasi-isomorphisms (or equivalently if $F\otimes_{\str_X} E$ is an exact complex for any exact complex of $\str_X$-modules $E$).
\end{defn}

\begin{lem}
There is a fully faithful, exact, coproduct preserving functor 
\begin{displaymath}
D(X) \to A_{\otimes}(\Inj X).
\end{displaymath}
\end{lem}
\begin{proof}
There is, by the proof of Theorem 5.5 of \cite{MurfetThesis}, a fully faithful, exact, coproduct preserving functor $D(X)\to N(\Flat X)$ given by taking K-flat resolutions and inducing an equivalence
\begin{displaymath}
D(X) \iso {}^{\perp}N_{\mathrm{ac}}(\Flat X).
\end{displaymath}
This functor given by taking resolutions factors via $A_\otimes(\Inj X)$ since K-flat complexes send acyclics to acyclics under the tensor product.
\end{proof}

Taking K-flat resolutions and then tensoring gives the desired action
\begin{displaymath}
(-) \odot (-) \colon D(X) \times S(X) \to S(X)
\end{displaymath}
 by an easy argument: K-flat resolutions are well behaved with respect to the tensor product so the necessary compatibilities follow from those of the tensor product of complexes.

\begin{rem}\label{rem_boundedabove}
Recall that every complex in $K^{-}(\Flat X)$, the homotopy category of bounded above complexes of flat $\str_X$-modules, is K-flat. 
 Thus when acting by the subcategory $K^{-}(\Flat X)$ there is an equality $\odot = \otimes_{\str_X}$.
\end{rem}

The tensor triangulated category $(D(X),\otimes,\str_X)$ is rigidly-compactly generated (where by $\otimes$ we of course mean the left derived tensor product). Thus we can apply all of the machinery we have developed for actions of rigidly-compactly generated triangulated categories. Recall from \cite{Thomclass} that $\Spc D(X)^c = \Spc D^{\mathrm{perf}}(X) \iso X$; henceforth we will identify these spaces. Via this identification we can associate to every specialization closed subset $\mathcal{V}$ of $X$ a localization sequence of submodules
\begin{displaymath}
\xymatrix{
\mathit{\Gamma}_\mathcal{V}S(X) \ar[r]<0.5ex> \ar@{<-}[r]<-0.5ex> & S(X) \ar[r]<0.5ex> \ar@{<-}[r]<-0.5ex>  & L_\mathcal{V}S(X),
}
\end{displaymath}
where $\mathit{\Gamma}_\mathcal{V}S(X)$ is generated by objects compact in $S(R)$, by Corollary A.4.11 and Lemma A.4.4.  Since $X$ is noetherian we get for every $x\in X$ objects $\mathit{\Gamma}_x\str_X$ which allow us to define supports on $S(X)$ with values in $X$; we will denote the support of an object $A$ of $S(X)$ simply by $\supp A$ as the action giving rise to this support will always be clear. 
We also wish to note that by Lemma A.4.6 the action restricts to the level of compact objects 
\begin{displaymath}
D^{\mathrm{perf}}(X) \times S(X)^{\mathrm{c}} \stackrel{\odot}{\to} S(X)^{\mathrm{c}}.
\end{displaymath}
and as the category $D(X)$ has a model the local-to-global principle (Theorem \ref{thm_general_ltg}) holds.

\subsection{Some first observations in the affine case}\label{ssec_affineactions}
Now let us restrict to the case $X = \Spec R$ for a noetherian ring $R$ and make a few simple observations about the $D(R)$ action. It is possible to give an explicit description of the Rickard idempotents associated to certain specialization closed subsets of $\Spec R$.  First we fix notation for the relevant complexes of $R$-modules.

\begin{defn}\label{notation_koszul}
Given an element $f\in R$ we define the \emph{stable Koszul complex} $K_\infty(f)$ to be the complex concentrated in degrees 0 and 1
\begin{displaymath}
\cdots \to 0 \to R \to R_f \to 0 \to \cdots
\end{displaymath}
where the only non-zero morphism is the canonical map to the localization. Given a sequence of elements $\mathbf{f} = \{f_1,\ldots,f_n\}$ of $R$ we set
\begin{displaymath}
K_\infty(\mathbf{f}) = K_\infty(f_1) \otimes \cdots \otimes K_\infty(f_n).
\end{displaymath}
We define the \emph{\^Cech complex} of $\mathbf{f}$ to be the suspension of the kernel of the canonical morphism $K(\mathbf{f}) \to R$. This is a degreewise split epimorphism and so we get a triangle in $K(A)$
\begin{displaymath}
K_\infty(\mathbf{f}) \to R \to \check{C}(\mathbf{f}) \to \S K_\infty(\mathbf{f}).
\end{displaymath}
Explicitly we have
\begin{displaymath}
\check{C}(\mathbf{f})^t = \bigoplus_{i_0<\cdots<i_t}R_{f_{i_0}\cdots f_{i_t}}
\end{displaymath}
for $0\leq t \leq n-1$ and $K_\infty(\mathbf{f})$ is degreewise the same complex desuspended and with $R$ in degree 0. For an ideal $I$ of $R$ we define $K(I)$ and $\check{C}(I)$ by choosing a set of generators for $I$; the complex obtained is independent of the choice of generators up to quasi-isomorphism in $D(R)$ and hence up to pure quasi-isomorphism in $A_\otimes(\Inj R)$. We note that these complexes are K-flat.
\end{defn}

Using these complexes we can give the following well known explicit descriptions of the Rickard idempotents corresponding to some specialization closed subsets.

\begin{prop}\label{prop_explicit_functors}
For an ideal $I\cie R$ and $\mathfrak{p}\in \Spec R$ a prime ideal there are natural isomorphisms in $D(R)$:
\begin{itemize}
\item[$(1)$] $\mathit{\Gamma}_{\mathcal{V}(I)}R \iso K_{\infty}(I)$;
\item[$(2)$] $L_{\mathcal{V}(I)}R \iso \check{C}(I)$;
\item[$(3)$] $L_{\mathcal{Z}(\mathfrak{p})}R \iso R_\mathfrak{p}$.
\end{itemize}
In particular the objects $\mathit{\Gamma}_\mathfrak{p}R = \mathit{\Gamma}_{\mathcal{V}(\mathfrak{p})}R\otimes L_{\mathcal{Z}(\mathfrak{p})}R$ giving rise to supports on $D(R)$ and $S(R)$ are naturally isomorphic to $K_{\infty}(\mathfrak{p})\otimes R_\mathfrak{p}$.
\end{prop}
\begin{proof}
Statements (1) and (2) are already essentially present in \cite{HartshorneLC}; in the form stated here they can be found as special cases of \cite{GreenleesTate} Lemma 5.8. For the third statement simply observe that the full subcategory of complexes with homological support in $\mathcal{U}(\mathfrak{p})$ is the essential image of the inclusion of $D(R_\mathfrak{p})$.
\end{proof}

We are now in a position to obtain, very cheaply, a couple of results about the singularity category and the action of $D(R)$ on it. We first observe that all localizing subcategories of $S(R)$ are $D(R)$-submodules.

\begin{lem}\label{lem_locpreserving}
Every localizing subcategory of $S(R)$ is stable under the action of $D(R)$.
\end{lem}
\begin{proof}
As $D(R)$ is the smallest localizing subcategory containing the tensor unit $R$ Lemma A.3.13 applies.
\end{proof}

Using the explicit description in Proposition \ref{prop_explicit_functors} of certain Rickard idempotents in $D(R)$ we are able to give representatives for the objects resulting from their action on objects of $S(R)$.

\begin{prop}\label{prop_subcomplex}
For each object $A$ of $S(R)$ and ideal $I\cie R$ the complex $\mathit{\Gamma}_{\mathcal{V}(I)}A$ is homotopic to a complex whose degree $i$ piece is the summand of $A^i$ consisting of those indecomposable injectives corresponding to primes in $\mathcal{V}(I)$.
\end{prop}
\begin{proof}
Let us fix an ideal $I$ and choose generators $I=(f_1,\ldots,f_n)$. By Proposition \ref{prop_explicit_functors} we have
\begin{displaymath}
\mathit{\Gamma}_{\mathcal{V}(I)}A \iso K_\infty(f_1,\ldots,f_n) \odot A \iso K_\infty(f_n)\odot(K_\infty(f_{n-1}) \odot \cdots (K_\infty(f_1)\odot A)\cdots ).
\end{displaymath}
We can thus reduce to the case that $I = (f)$. By Proposition \ref{prop_explicit_functors} again we have 
\begin{displaymath}
L_{\mathcal{V}((f))}A \iso \check{C}(f)\otimes_R A \iso R_f \otimes_R A
\end{displaymath}
where the last isomorphism uses the explicit description of the \^Cech complex given in Definition \ref{notation_koszul}.  The canonical map $A\to R_f\otimes_R A$ is a degreewise split epimorphism in the category of chain complexes which fits into the localization triangle
\begin{displaymath}
K_\infty(f)\otimes_R A \to A \to R_f \otimes_R A \to \S K_{\infty}(f)\otimes_R A
\end{displaymath}
in $S(R)$.  So up to homotopy $K_\infty(f)\odot A$ is the kernel of this split epimorphism. The kernel in each degree is precisely the summand consisting of those indecomposable injectives corresponding to primes in $\mathcal{V}(I)$ which proves the claim.
\end{proof}

\section{Some commutative algebra}\label{sec_prelims_alg}
We give here a brief summary of some commutative algebra definitions and results concerning Gorenstein homological algebra and local complete intersections. Our main reference for Gorenstein homological algebra is \cite{EnochsJenda}, particularly Chapters 10 and 11. 

\subsection{Gorenstein Homological Algebra}\label{ssec_GInj}

Let us denote by $R$ a noetherian ring.

\begin{defn}\label{defn_ginj}
An $R$-module $G$ is \emph{Gorenstein injective} if there exists an exact sequence
\begin{displaymath}
E = \cdots \to E_1 \to E_0 \to E^0 \to E^1 \to \cdots
\end{displaymath}
of injective $R$-modules such that $\Hom(I,E)$ is exact for every injective module $I$ and $G = \ker( E^0 \to E^1) = Z^0E$ is the zeroth syzygy of $E$. We say that the complex $E$ is \emph{totally acyclic} and call it a \emph{complete resolution} of $G$. We denote the full subcategory of Gorenstein injective $R$-modules by $\GInj R$.
\end{defn}

\begin{defn}
We denote by $K_{\mathrm{tac}}(\Inj R)$ the homotopy category of totally acyclic complexes of injective $R$-modules. It is a full triangulated subcategory of $S(R) = K_{\mathrm{ac}}(\Inj R)$ and the two coincide when $R$ is Gorenstein by \cite{KrStab} Proposition 7.13.
\end{defn}

We now recall the notion of envelopes with respect to a class of $R$-modules.

\begin{defn}
Let $R$ be a ring and fix some class $\mathcal{G}$ of $R$-modules. A $\mathcal{G}$-\emph{preenvelope} of an $R$-module $M$ is a pair $(G,f)$ where $G$ is a module in $\mathcal{G}$ and $f\colon M \to G$ is a morphism such that for any $G' \in \mathcal{G}$ and $f'\in \Hom(M,G')$ there exists a morphism making the triangle
\begin{displaymath}
\xymatrix{
M \ar[r]^f \ar[dr]_{f'} & G \ar@{-->}[d]^{\exists} \\
& G'
}
\end{displaymath}
commute. We say that a preenvelope $(G,f)$ is a $\mathcal{G}$-\emph{envelope} of $M$ if, when we consider the diagram
\begin{displaymath}
\xymatrix{
M \ar[r]^f \ar[dr]_{f} & G \ar@{-->}[d]^{\exists} \\
& G
}
\end{displaymath}
every choice of morphism $G\to G$ making the triangle commute is an automorphism of $G$. When speaking of envelopes we shall omit the morphism from the notation and refer to $G$ as the $\mathcal{G}$-envelope of $M$.
\end{defn}

It is clear that $\mathcal{G}$-envelopes, when they exist, are unique up to isomorphism. In the case $\mathcal{G} = \Inj R$, the class of injective $R$-modules, we see that the $\Inj R$-envelope of an $R$-module is precisely its injective envelope. In the case $\mathcal{G} = \GInj R$, the class of Gorenstein injective $R$-modules, we get the notion of \emph{Gorenstein injective envelope}. As every injective $R$-module is Gorenstein injective it can be shown that whenever $(G,f)$ is a Gorenstein injective envelope of $M$ the morphism $f$ is injective.

\begin{notation}
For an $R$-module $M$ we shall denote its Gorenstein injective envelope, if it exists, by $G_R(M)$.
\end{notation}

In many cases Gorenstein injective envelopes exist and have certain nice properties.

\begin{thm}[\cite{EnochsJenda} 11.3.2, 11.3.3]\label{thm_env_exist}
If $R$ is Gorenstein of Krull dimension $n$, then every $R$-module $M$ admits a Gorenstein injective envelope $M\to G$ such that if
\begin{displaymath}
0 \to M \to G \to L \to 0
\end{displaymath}
is exact then $\idim_R L \leq n-1$ whenever $n \geq 1$. Furthermore, $\idim_R M < \infty$ if and only if $M\to G$ is an injective envelope.
\end{thm}

\begin{prop}[\cite{EnochsJenda} 10.4.25, 11.3.9]\label{prop_ginj_dsums}
Let $R$ be a Gorenstein ring of finite Krull dimension. Then the class of Gorenstein injective $R$-modules is closed under small coproducts and summands and if $M_i \to G_i$ is a Gorenstein injective envelope of the $R$-module $M_i$ for each $i\in I$, then
\begin{displaymath}
\oplus_i M_i \to \oplus_i G_i
\end{displaymath}
is a Gorenstein injective envelope.
\end{prop}

As mentioned above we denote by $\GInj R$ the category of Gorenstein injective $R$-modules. It is a Frobenius category i.e., it is an exact category with enough projectives and enough injectives and the projective and injective objects coincide. The exact structure comes from taking the exact sequences to be those exact sequences of $R$-modules whose terms are Gorenstein injective. 

The \emph{stable category} of $\GInj R$ denoted $\underline{\GInj}R$ is the category whose objects are those of $\GInj R$ and whose hom-sets are
\begin{align*}
\underline{\Hom}(G,H) &:= \Hom_{\underline{\GInj} R}(G,H) \\
&\;= \Hom_R(G,H)/\{f \; \vert \; f \; \text{factors via an injective module}\}.
\end{align*}
This category is triangulated with suspension functor given by taking syzygies of complete resolutions and triangles coming from short exact sequences (see for example \cite{Happeltricat} Chapter 1 for details).

The following result shows that we can study part of the singularity category $S(R)$ by working with Gorenstein injective $R$-modules.

\begin{prop}\label{prop_ginj_sing_equiv}
For a noetherian ring $R$ there is an equivalence
\begin{displaymath}
\xymatrix{
K_{\mathrm{tac}}(\Inj R) \ar[r]<0.5ex>^(0.6){Z^0} \ar@{<-}[r]<-0.5ex>_(0.6){c} & \GInj \underline{R}
}
\end{displaymath}
where $Z^0$ takes the zeroth syzygy of a complex of injectives and $c$ sends a Gorenstein injective $R$-module to a complete resolution.
\end{prop}
\begin{proof}
The result is standard so we only sketch the proof. The functor $c$ is well defined since complete resolutions are unique up to homotopy equivalence, exist by definition for every Gorenstein injective module, and morphisms of modules lift uniquely up to homotopy to morphisms of complete resolutions. The zeroth syzygy of any totally acyclic complex of injectives is by definition Gorenstein injective. It is clear that, up to injectives, the Gorenstein injective $R$-module obtained by applying $Z^0$ to a totally acyclic complex of injectives only depends on its homotopy class so that $Z^0$ is well defined.

It is easy to check that the requisite composites are naturally isomorphic to the corresponding identity functors.
\end{proof}

For $R$ Gorenstein every acyclic complex of injectives is totally acyclic by \cite{KrStab} Proposition 7.13 so there is an equivalence between $S(R)$ and $\underline{\GInj}R$. We thus obtain an action of $D(R)$ on $\underline{\GInj} R$ via this equivalence and we use the same notation to denote this action.

\begin{cor}\label{cor_envelope}
Let $R$ be a Gorenstein ring and $M$ an $R$-module. There is an isomorphism in $\underline{\GInj} R$
\begin{displaymath}
Z^0 I_\l Q_\r M \iso G_R(M).
\end{displaymath}
\end{cor}
\begin{proof}
By Theorem \ref{thm_env_exist} there is a short exact sequence of $R$-modules
\begin{displaymath}
0 \to M \to G_R(M) \to L \to 0
\end{displaymath}
where $L$ has finite projective dimension. Considering this as a triangle in $D(R)$ we obtain a triangle in $\underline{\GInj}R$
\begin{displaymath}
Z^0I_\l Q_\r M \to Z^0 I_\l Q_\r G_R(M) \to 0 \to \S Z^0 I_\l Q_\r M
\end{displaymath}
where $L$ is sent to zero as it is a perfect complex when viewed as an object of $D(R)$. By \cite{KrStab} Corollary 7.14 the object $Z^0 I_\l Q_\r G_R(M)$ is naturally isomorphic to the image of $G_R(M)$ in $\underline{\GInj} R$ under the canonical projection which proves the claim.
\end{proof}

\subsection{Complete Intersections and Complexity}\label{ssec_cx}
We now give a very brief recollection of some facts about local complete intersections and growth rates of minimal free resolutions over local rings. First of all let us recall the definition of the rings which will be of most interest to us.

\begin{defn}\label{defn_CIring}
Let $(R,\mathfrak{m},k)$ be a noetherian local ring. We say $R$ is a \emph{complete intersection} if its $\mathfrak{m}$-adic completion $\hat{R}$ can be written as the quotient of a regular ring by a regular sequence. The minimal length of the regular sequence in such a presentation of $\hat{R}$ is the \emph{codimension} of $R$.

A not necessarily local ring $R$ is a locally complete intersection if $R_\mathfrak{p}$ is a complete intersection for each  $\mathfrak{p}\in \Spec R$.

If $R$ is a complete intersection of codimension $1$ we say that it is a \emph{hypersurface}. Similarly if $R$ is a complete intersection of codimension at most $1$ when localized at each prime ideal in $\Spec R$ we say that $R$ is \emph{locally a hypersurface}.
\end{defn}

\begin{rem}
Rings satisfying the conditions of the above definition are sometimes called \emph{abstract} complete intersections to differentiate them from those local rings which are quotients of regular rings by regular sequences without the need to complete. We use the term complete intersection as in the definition above. When we need to impose that $R$ itself is a quotient of a regular ring by a regular sequence we will make this clear.
\end{rem}

\begin{rem}
The property of being a complete intersection is stable under localization. Furthermore, the property of being a complete intersection is intrinsic (cf.\ \cite{MatsuRing} Theorem 21.2) and if $(R,\mathfrak{m},k)$ is a complete intersection then any presentation of $\hat{R}$ as a quotient of a regular local ring has kernel generated by a regular sequence.
\end{rem}

Let $(R,\mathfrak{m},k)$ be a local ring. Given a finitely generated $R$-module $M$ we denote by $\b_i(M)$ the $i$th \emph{Betti number} of $M$
\begin{displaymath}
\b_i(M) = \dim_k \Tor_i(M,k) = \dim_k \Ext^i(M,k).
\end{displaymath}

The asymptotic behaviour of the Betti numbers is expressed by the \emph{complexity} of $M$. For $M$ in $R\text{-}\module$ the \emph{complexity} of $M$, $\cx_R(M)$ (or just $\cx M$ if the ring is clear), is defined to be
\begin{displaymath}
\cx(M) = \inf\{c\in \nat \; \vert \; \text{there exists}\; a\in \real \; \text{such that} \; \b_n(M) \leq an^{c-1} \;\; \text{for} \; n\gg 0\}.
\end{displaymath}


By a result of Gulliksen (\cite{Gulliksen} Theorem 2.3) the complexity of the residue field $k$ detects whether or not $R$ is a complete intersection. If $R$ is a complete intersection then $\cx_R k$ is equal to the codimension of $R$. 

\section{Affine schemes}\label{sec_affine}
Recall from Definition \ref{defn_vis_sigmatau} that the action of $D(R)$ on $S(R)$ gives rise to order preserving assignments
\begin{displaymath}
\left\{ \begin{array}{c}
\text{subsets of}\; \Spec R 
\end{array} \right\}
\xymatrix{ \ar[r]<1ex>^\t \ar@{<-}[r]<-1ex>_\s &} \left\{
\begin{array}{c}
\text{localizing subcategories} \; \text{of} \; S(R) \\
\end{array} \right\} 
\end{displaymath}
where for a localizing subcategory $\mathcal{L}$ we set
\begin{displaymath}
\s(\mathcal{L}) = \supp \mathcal{L} = \{\mathfrak{p} \in \Spec R \; \vert \; \mathit{\Gamma}_\mathfrak{p}\mathcal{L} \neq 0\}
\end{displaymath}
and for a subset $W\cie \Spec R$ we set
\begin{displaymath}
\t(W) = \{A \in S(R) \; \vert \; \supp A \cie W\}.
\end{displaymath}
Here we have used Lemma \ref{lem_locpreserving} to replace submodules by localizing subcategories. As $D(R)$ satisfies the local-to-global principle one can say a little more.

\begin{prop}\label{prop_imtau}
Given a subset $W\cie \Spec R$ there is an equality of subcategories
\begin{displaymath}
\t(W) = \langle \mathit{\Gamma}_\mathfrak{p}S(R) \; \vert \; \mathfrak{p}\in W\rangle_\mathrm{loc}.
\end{displaymath}
\end{prop}
\begin{proof}
This is just a restatement of Lemma A.6.2.

\end{proof}

We next note that, as one would expect, $S(R)$ is supported on the singular locus $\Sing R$ of $\Spec R$.

\begin{lem}
There is a containment $\s S(R) \cie \Sing R$.
\end{lem}
\begin{proof}
If $\mathfrak{p}\in \Spec R$ is a regular point then $S(R_\mathfrak{p}) = 0$. Thus for any object $A$ of $S(R)$
\begin{displaymath}
\mathit{\Gamma}_\mathfrak{p} A \iso R_\mathfrak{p}\otimes_R (\mathit{\Gamma}_{\mathcal{V}(\mathfrak{p})}R\odot A) \iso 0
\end{displaymath}
as it is an acyclic complex of injective $R_\mathfrak{p}$-modules.
\end{proof}

It is clear that $D(R)$ also acts, by K-flat resolutions, on itself and on $K(\Inj R)$. It will be convenient for us to show that these actions are compatible with each other and the action on $S(R)$ in an appropriate sense. We write $\otimes$ for the action of $D(R)$ on itself and $\odot$ for the action of $D(R)$ on $K(\Inj R)$ which extends the action on $S(R)$.

\begin{prop}\label{prop_action_compatible}
These actions of $D(R)$ are compatible with the localization sequence
\begin{displaymath}
\xymatrix{
D(R) \ar[rr]^{Q_\l} && K(\Inj R) \ar[rr]^{I_\l} \ar@/^1pc/[ll]^{Q} &&  S(R) \ar@/^1pc/[ll]^{I}
}
\end{displaymath}
in the sense that, up to natural isomorphism, the action commutes with each of the functors in the diagram. Explicitly, for $J\in K(\Inj R)$, $A\in S(R)$, and $E,F\in D(R)$ we have isomorphisms
\begin{displaymath}
\begin{array}{ccc}
Q(E\odot J) \iso E\otimes QJ &, & Q_\l(E\otimes F) \iso E\odot Q_\l F  \\
\\
I(E\odot A) \iso E\odot IA  &, & I_\l(E\odot J) \iso E\odot I_\l J.
\end{array}
\end{displaymath}
\end{prop}
\begin{proof}
It is obvious that the inclusion $I$ is compatible with the action of $D(R)$. As $D(R)$ acts on $K(\Inj R)$ via K-flat resolutions the action commutes with $Q$; an object $J$ of $K(\Inj R)$ is quasi-isomorphic to $QJ$ so for $E\in D(R)$ the object $E\odot J$ computes the left derived tensor product.

To treat the other two functors let $J$ be an object of $K(\Inj R)$ and consider the localization triangle
\begin{displaymath}
Q_\l QJ \to J \to I I_\l J \to \S Q_\l QJ.
\end{displaymath}
As $D(R)$ is generated by the tensor unit $R$ every localizing subcategory of $K(\Inj R)$ is stable under the action (cf.\ Lemma \ref{lem_locpreserving}). Thus for $E\in D(R)$ we get a triangle,
\begin{displaymath}
E\odot Q_\l QJ \to E\odot J \to E\odot I I_\l J \to \S E\odot Q_\l QJ,
\end{displaymath}
where $E\odot Q_\l QJ \in Q_\l D(R)$ and $E\odot I I_\l J$ is acyclic. Hence this triangle must be uniquely isomorphic to the localization triangle for $E\odot J$ giving
\begin{displaymath}
E\odot Q_\l QJ \iso Q_\l Q(E\odot J) \quad \quad \text{and} \quad \quad E\odot I I_\l J \iso I I_\l (E\odot J).
\end{displaymath}
We already know that the action commutes with $Q$ and $I$ so the remaining two commutativity relations follow immediately.
\end{proof}

We can also say something about compatibility with the right adjoint $Q_\r$ of $Q$.

\begin{lem}\label{lem_action_compatible}
Suppose $E$ and $F$ are objects of $D(R)$ such that $E$ has a bounded flat resolution and $F$ has a bounded below injective resolution. Then 
\begin{displaymath}
E\odot Q_\r F \iso Q_\r(E\otimes F).
\end{displaymath}
\end{lem}
\begin{proof}
Let $\tilde{E}$ be a bounded flat resolution of $E$ and recall that, by virtue of being bounded, $\tilde{E}$ is K-flat. In \cite{KrStab} Krause identifies $Q_\r$ with taking K-injective resolutions, where the K-injectives are the objects of the colocalizing subcategory of $K(\Inj R)$ generated by the bounded below complexes of injectives (such resolutions exist, see for example \cite{NeeHolim}). Thus $Q_\r F$ is a K-injective resolution of $F$ and so, as bounded below complexes of injectives are K-injective, we may assume it is bounded below as it is homotopic to the bounded below resolution we have required of $F$. We have $Q_\r F \iso F$ in $D(R)$ so there are isomorphisms in the derived category
\begin{displaymath}
E\otimes F \iso E\otimes Q_\r F \iso \tilde{E} \otimes_R Q_\r F.
\end{displaymath}
Hence in $K(\Inj R)$ we have isomorphisms
\begin{align*}
Q_\r(E\otimes F) &\iso Q_\r(E\otimes Q_\r F) \\
&\iso Q_\r(\tilde{E} \otimes_R Q_\r F) \\
&\iso \tilde{E}\otimes_R Q_\r F \\
&\iso E\odot Q_\r F
\end{align*}
where the penultimate isomorphism is a consequence of the fact that $\tilde{E}\otimes_R Q_\r F$ is, by the assumption that $\tilde{E}$ is bounded and $Q_\r F$ is bounded below, a bounded below complex of injectives and hence K-injective.


\end{proof}

Before proceeding let us record the following easy observation for later use.

\begin{lem}\label{lem_canlocalize}
The diagram
\begin{displaymath}
\xymatrix{
D^+(R) \ar[d] \ar[r]^{I_\l Q_\r} & S(R) \ar[d] \\
D^+(R_\mathfrak{p}) \ar[r]_{I_\l Q_\r} & S(R_\mathfrak{p}),
}
\end{displaymath}
where the vertical functors are localization at $\mathfrak{p}\in \Spec R$, commutes.
\end{lem}
\begin{proof}
The square
\begin{displaymath}
\xymatrix{
D^+(R) \ar[d] \ar[r]^{Q_\r} & K(\Inj R) \ar[d] \\
D^+(R_\mathfrak{p}) \ar[r]_{Q_\r} & K(\Inj R_\mathfrak{p}),
}
\end{displaymath}
is commutative by Lemma \ref{lem_action_compatible}.

To complete the proof it is sufficient to show that the square
\begin{displaymath}
\xymatrix{
K(\Inj R) \ar[d] \ar[r]^(0.6){I_\l} & S(R) \ar[d] \\
K(\Inj R_\mathfrak{p}) \ar[r]_(0.6){I_\l} & S(R_\mathfrak{p}),
}
\end{displaymath}
also commutes. This follows by observing that the square
\begin{displaymath}
\xymatrix{
K(\Inj R) & \ar[l]_(0.4)I S(R) \\
K(\Inj R_\mathfrak{p}) \ar[u]& \ar[u] \ar[l]^(0.4)I S(R_\mathfrak{p}),
}
\end{displaymath}
commutes and taking left adjoints, where we are using the fact that tensoring and restricting scalars along $R\to R_\mathfrak{p}$ are both exact and preserve injectives so give rise to an adjoint pair or functors between the relevant homotopy categories of injectives and singularity categories.
\end{proof}

\begin{rem}
One cannot expect the last result to hold for unbounded complexes. Indeed in the recent paper \cite{ChenIyengarCEX} an example of a K-injective complex whose localization at some prime is not K-injective is given. I am grateful to the anonymous referees of my PhD thesis for pointing this out.
\end{rem}

Given these compatibilities it is not hard to see that $\s S(R)$ is precisely the singular locus.

\begin{prop}\label{prop_nontriv}
For any $\mathfrak{p}\in \Sing R$ the object $\mathit{\Gamma}_\mathfrak{p}I_\l Q_\r k(\mathfrak{p})$ is non-zero in $S(R)$. Thus $\mathit{\Gamma}_\mathfrak{p}S(R)$ is non-trivial for all such $\mathfrak{p}$ yielding the equality ~$\s S(R) = \Sing R$.
\end{prop}
\begin{proof}
Let $\mathfrak{p}$ be a singular point of $\Spec R$. Applying the last lemma we may check that $I_\l Q_\r k(\mathfrak{p})$ is non-zero over $R_\mathfrak{p}$.  By \cite{MatsuAlg} Section 18 Theorem 41 one has
\begin{displaymath}
\pdim_{R_{\mathfrak{p}}}k(\mathfrak{p}) = \gldim R_{\mathfrak{p}} = \infty
\end{displaymath}
so $k(\mathfrak{p})$ is finitely generated over $R_\mathfrak{p}$ but not perfect. Theorem \ref{thm_recol} then tells us that $I_\l Q_\r k(\mathfrak{p})$ is not zero in $S(R_\mathfrak{p})$. 

We now show $I_\l Q_\r k(\mathfrak{p})$ lies in $\mathit{\Gamma}_\mathfrak{p}S(R)$. By Proposition \ref{prop_action_compatible} there is an isomorphism
\begin{equation}\label{eq_nontrivial}
\mathit{\Gamma}_\mathfrak{p}R \odot I_\l Q_\r k(\mathfrak{p}) \iso I_\l (\mathit{\Gamma}_\mathfrak{p}R\odot Q_\r k(\mathfrak{p})).
\end{equation}
As $\mathit{\Gamma}_\mathfrak{p}R \iso K_\infty(\mathfrak{p})\otimes R_\mathfrak{p}$ (by Proposition \ref{prop_explicit_functors}) is a bounded K-flat complex and the injective resolution of $k(\mathfrak{p})$ is certainly bounded below we can apply Lemma \ref{lem_action_compatible}. This gives us isomorphisms
\begin{displaymath}
I_\l (\mathit{\Gamma}_\mathfrak{p}R \odot Q_\r k(\mathfrak{p})) \iso I_\l Q_\r (\mathit{\Gamma}_\mathfrak{p}R \otimes k(\mathfrak{p})) \iso I_\l Q_\r k(\mathfrak{p}).
\end{displaymath}
Combining these with (\ref{eq_nontrivial}) shows that, up to homotopy, $\mathit{\Gamma}_\mathfrak{p}R\odot(-)$ fixes the non-zero object $I_\l Q_\r k(\mathfrak{p})$ proving that $\mathit{\Gamma}_\mathfrak{p}S(R)$ is non-zero.
\end{proof}

It is thus natural to restrict the support and the assignments $\s$ and $\t$ to subsets of the singular locus. This result then implies that the assignment $\t$ taking a subset of $\Sing R$ to a localizing subcategory of $S(R)$ is injective with left inverse $\s$.

\begin{cor}\label{cor_injective}
For every $W\cie \Sing R$ we have $\s\t(W) = W$. In particular, $\t$ is injective when restricted to subsets of the singular locus.
\end{cor}
\begin{proof}
This is just an application of Proposition A.6.3.
\end{proof}

We now prove some results concerning generators for the subcategories produced via the action of $D(R)$. This will allow us to describe the image of $\t$ as the localizing subcategories which contain certain objects.

The next lemma is an easy modification of an argument of Krause in \cite{KrStab}. We give the details, including those straight from Krause's proof, as it is clearer to present them along with the modifications than to just indicate what else needs to be checked.

\begin{lem}\label{lem_gen1}
Let $\mathcal{V}$ be a specialization closed subset of $\Sing R$. The set of objects
\begin{displaymath}
\{\S^i I_\l Q_\r R/\mathfrak{p} \; \vert \; \mathfrak{p}\in \mathcal{V}, i\in \int\}
\end{displaymath}
is a generating set for $\mathit{\Gamma}_\mathcal{V}S(R)$ consisting of objects which are compact in $S(R)$.
\end{lem}
\begin{proof}
%
Let $A$ be a non-zero object of $\mathit{\Gamma}_\mathcal{V}S(R)$. In particular $A$ is a complex of injectives satisfying $H^nA = 0$ for all $n\in \int$. As $A$ is not nullhomotopic we can choose $n$ such that $Z^nA$ is not injective. Consider the beginning of an augmented minimal injective resolution of $Z^nA$
\begin{displaymath}
\xymatrix{
0 \ar[r] & Z^nA \ar[r] & E^0(Z^nA) \ar[r] & E^1(Z^nA).
}
\end{displaymath}
Note that for $\mathfrak{q}\notin \mathcal{V}$ the object $R_\mathfrak{q}\otimes \mathit{\Gamma}_\mathcal{V}R$ is zero in $D(R)$ as the cohomology of $\mathit{\Gamma}_\mathcal{V}R$ is supported in $\mathcal{V}$ by definition.  Thus for $\mathfrak{q}\notin \mathcal{V}$ the complex $A_\mathfrak{q}$ is nullhomotopic by virtue of being in the essential image of $\mathit{\Gamma}_\mathcal{V}R\odot(-)$. So $Z^nA_\mathfrak{q}$ is injective as a nullhomotopic complex is split exact. Since, for modules, localization at a prime sends minimal injective resolutions to minimal injective resolutions (see for example \cite{MatsuRing} Section 18) for any such $\mathfrak{q}$ it holds that $E^1(Z^nA)_\mathfrak{q} = 0$. So writing
\begin{displaymath}
E^1(Z^nA) \iso \bigoplus_i E(R/\mathfrak{p}_i)
\end{displaymath}
we have $\mathfrak{p}\in \mathcal{V}$ for each distinct $\mathfrak{p}$ occurring in the direct sum as otherwise it would not vanish when localized (see for example \cite{BIK} Lemma 2.1). Now fix some $\mathfrak{p}$ such that $E(R/\mathfrak{p})$ occurs in $E^1(Z^nA)$. By \cite{EnochsJenda} Theorem 9.2.4, as the injective envelope of $\mathfrak{p}$ occurs in $E^1(Z^nA)$, we have
\begin{displaymath}
0 \neq \dim_{k(\mathfrak{p})}\Ext^1(k(\mathfrak{p}), Z^nA_\mathfrak{p}) = \dim_{k(\mathfrak{p})}\Ext^1(R/\mathfrak{p}, Z^nA)_\mathfrak{p}.
\end{displaymath}
In particular $\Ext^1(R/\mathfrak{p}, Z^nA)$ is non-zero. Using \cite{KrStab} Lemma 2.1 and the adjunction between $I$ and $I_\l$ there are isomorphisms
\begin{align*}
\Ext^1(R/\mathfrak{p}, Z^nA) &\iso \Hom_{K(R\text{-}\Module)}(R/\mathfrak{p}, \S^{n+1} IA) \\
&\iso \Hom_{K(\Inj R)}(Q_\r R/\mathfrak{p}, \S^{n+1} IA) \\
&\iso \Hom_{S(R)}(\S^{-n-1}I_\l Q_\r R/\mathfrak{p}, A).
\end{align*}
Thus the set in question is certainly generating and it consists of compact objects by Theorem \ref{thm_recol} (3).
\end{proof}

\begin{lem}\label{lem_gen2}
The object $I_\l Q_\r k(\mathfrak{p})$ generates $\mathit{\Gamma}_\mathfrak{p}S(R)$ for every $\mathfrak{p}\in \Sing R$ i.e,
\begin{displaymath}
\mathit{\Gamma}_\mathfrak{p}S(R) = \langle I_\l Q_\r k(\mathfrak{p})\rangle_\mathrm{loc}.
\end{displaymath}
\end{lem}
\begin{proof}
By Lemma \ref{lem_gen1} we have an equality
\begin{displaymath}
\mathit{\Gamma}_{\mathcal{V}(\mathfrak{p})}S(R) = \langle I_\l Q_\r R/\mathfrak{q}\; \vert \; \mathfrak{q}\in \mathcal{V}(\mathfrak{p})\rangle_\mathrm{loc}.
\end{displaymath}
Noticing that
\begin{displaymath}
\mathit{\Gamma}_\mathfrak{p}S(R) = L_{\mathcal{Z}(\mathfrak{p})}\mathit{\Gamma}_{\mathcal{V}(\mathfrak{p})}S(R) = \langle L_{\mathcal{Z}(\mathfrak{p})}R \rangle_\mathrm{loc} \odot \mathit{\Gamma}_{\mathcal{V}(\mathfrak{p})}S(R)
\end{displaymath}
we thus get, by Lemma A.3.12, equalities
\begin{align*}
\langle L_{\mathcal{Z}(\mathfrak{p})}R \rangle_\mathrm{loc} \odot \mathit{\Gamma}_{\mathcal{V}(\mathfrak{p})}S(R) &= \langle R_\mathfrak{p} \rangle_\mathrm{loc} \odot \langle I_\l Q_\r R/\mathfrak{q}\; \vert \; \mathfrak{q}\in \mathcal{V}(\mathfrak{p})\rangle_\mathrm{loc} \\
&= \langle R_\mathfrak{p} \odot I_\l Q_\r R/\mathfrak{q}\; \vert \; \mathfrak{q}\in \mathcal{V}(\mathfrak{p})\rangle_\mathrm{loc}
\end{align*}
where we have used Proposition \ref{prop_explicit_functors} to identify $L_{\mathcal{Z}(\mathfrak{p})}R$ with $R_\mathfrak{p}$. Hence, using Proposition \ref{prop_action_compatible} and Lemma \ref{lem_action_compatible} to move the action by $R_\mathfrak{p}$ past $I_\l Q_\r$, we obtain equalities
\begin{displaymath}
\mathit{\Gamma}_\mathfrak{p}S(R) = \langle I_\l Q_\r(R_\mathfrak{p} \otimes R/\mathfrak{q})\; \vert \; \mathfrak{q}\in \mathcal{V}(\mathfrak{p})\rangle_\mathrm{loc} = \langle I_\l Q_\r k(\mathfrak{p})\rangle_\mathrm{loc}
\end{displaymath}
completing the proof.
\end{proof}

Next we consider the behaviour of $\s$ and $\t$ with respect to the collection of subcategories of $S(R)$ generated by objects of $S(R)^c$.  The key observation is that when $R$ is Gorenstein compact objects of $S(R)$ have closed supports. In order to prove this we first show that one can reduce to considering the images of modules.

\begin{lem}\label{lem_reduction}
Let $a$ be a compact object of $S(R)$. Then there exists a finitely generated $R$-module $M$ and integer $i$ such that $a\oplus \S a$ is isomorphic to $\S^i I_\l Q_\r M$. In particular there is an equality
\begin{displaymath}
\supp a = \supp I_\l Q_\r M.
\end{displaymath}
\end{lem}
\begin{proof}
By Theorem \ref{thm_recol} $I_\l Q_\r$ induces an equivalence up to summands between $D_{\mathrm{Sg}}(R)$ and $S(R)^c$ so $a\oplus \S a$ is in the image of $I_\l Q_\r$ by \cite{NeeCat} Corollary 4.5.12. By the argument of \cite{OrlovSing} Lemma 1.11 every object of $D_\mathrm{Sg}(R)$ is, up to suspension, the image of a finitely generated $R$-module so we can find a finitely generated $M$ as claimed.

The statement about supports follows from the properties of the support given in Proposition A.5.7.
\end{proof}

\begin{lem}\label{lem_bigsupp}
If $a$ is an object of $S(R)^c$ then the set
\begin{displaymath}
\{\mathfrak{p}\in \Sing R\; \vert \; L_{\mathcal{Z}(\mathfrak{p})}R\odot a \neq 0\}
\end{displaymath}
is closed in $\Spec R$.
\end{lem}
\begin{proof}
Clearly we may, by applying Lemma \ref{lem_reduction}, suppose $a$ is $I_\l Q_\r M$ where $M$ is a finitely generated $R$-module. By the compatibility conditions of \ref{prop_action_compatible} and \ref{lem_action_compatible} we have an isomorphism
\begin{displaymath}
L_{\mathcal{Z}(\mathfrak{p})}R \odot I_\l Q_\r M \iso I_\l Q_\r M_\mathfrak{p}
\end{displaymath}
as $L_{\mathcal{Z}(\mathfrak{p})}R \iso R_\mathfrak{p}$ by Proposition \ref{prop_explicit_functors}. 

By considering the diagram of Lemma \ref{lem_canlocalize} and noting that the module $M_\mathfrak{p}$ is finitely generated over $R_\mathfrak{p}$ we see the object $I_\l Q_\r M_\mathfrak{p}$ is zero precisely when $M_\mathfrak{p}$ has finite projective dimension.

Thus
\begin{displaymath}
\{\mathfrak{p}\in \Sing R\; \vert \; L_{\mathcal{Z}(\mathfrak{p})}R\odot I_\l Q_\r M \neq 0\} = \{\mathfrak{p}\in \Sing R \; \vert \; \pdim_{R_\mathfrak{p}}M_\mathfrak{p} = \infty\}
\end{displaymath}
and this latter set is closed as $M$ is finitely generated.
\end{proof}

\begin{lem}\label{lem_compactsupp}
Let $R$ be Gorenstein and let $a$ be a compact object of $S(R)$. Then $\mathfrak{p}\in \Sing R$ is in the support of $a$ if and only if $L_{\mathcal{Z}(\mathfrak{p})}R\odot a$ is not zero.
\end{lem}
\begin{proof}
One direction is easy: if $\mathfrak{p}\in \supp a$ then
\begin{displaymath}
\mathit{\Gamma}_\mathfrak{p}R \odot a \iso \mathit{\Gamma}_{\mathcal{V}(\mathfrak{p})}R \odot L_{\mathcal{Z}(\mathfrak{p})}R \odot a \neq 0
\end{displaymath}
so $L_{\mathcal{Z}(\mathfrak{p})}R\odot a$ is certainly not zero.

Now let us prove the converse. By Lemma \ref{lem_reduction} it is sufficient to prove the result for $I_\l Q_\r M$ where $M$ is a finitely generated $R$-module. So suppose $M$ is a finitely generated $R$-module of infinite projective dimension such that the projection of $R_\mathfrak{p}\otimes_R M = M_\mathfrak{p}$ to $S(R_\mathfrak{p})^c$ is not zero, where this projection is  $L_{\mathcal{Z}(\mathfrak{p})}I_\l Q_\r M$ by the compatibility conditions of \ref{prop_action_compatible} and \ref{lem_action_compatible}. In particular $M_\mathfrak{p}$ also has infinite projective dimension. 

As $R_\mathfrak{p}$ is Gorenstein of finite Krull dimension $M_\mathfrak{p}$ has, as an $R_\mathfrak{p}$-module, a Gorenstein injective envelope $G(M_\mathfrak{p})$ by Theorem \ref{thm_env_exist} which fits into an exact sequence
\begin{displaymath}
0 \to M_\mathfrak{p} \to G(M_\mathfrak{p}) \to L \to 0
\end{displaymath}
where $L$ has finite injective dimension (details about Gorenstein injectives and Gorenstein injective envelopes can be found in Section \ref{ssec_GInj}). So for $i$ sufficiently large (i.e., exceeding the dimension of $R_\mathfrak{p}$) we have isomorphisms
\begin{align*}
&\;\;\;\; \Hom_{S(R)}(I_\l Q_\r R/\mathfrak{p}, \S^i L_{\mathcal{Z}(\mathfrak{p})} I_\l Q_\r M) \\
&\iso \Hom_{S(R_\mathfrak{p})}(I_\l Q_\r k(\mathfrak{p}), \S^i I_\l Q_\r M_\mathfrak{p}) \\
&\iso \Hom_{S(R_\mathfrak{p})}(I_\l Q_\r k(\mathfrak{p}), \S^i I_\l Q_\r G(M_\mathfrak{p})) \\
&\iso \Ext^i_{R_\mathfrak{p}}(k(\mathfrak{p}), G(M_\mathfrak{p})) \\
&\iso \Ext^i_{R_\mathfrak{p}}( k(\mathfrak{p}), M_\mathfrak{p})
\end{align*}
where the first isomorphism is by adjunction, the second by the identification of a complete injective resolution for $M_\mathfrak{p}$ with the defining complex of $G(M_\mathfrak{p})$ (see Proposition \ref{prop_ginj_sing_equiv} and Corollary \ref{cor_envelope}, cf.\  \cite{KrStab} Section 7), the third by \cite{KrStab} Proposition 7.10, and the last isomorphism by the finiteness of the injective dimension of $L$. 

From \cite{EnochsJenda} Proposition 9.2.13 we learn that for $\mathfrak{q} \cie \mathfrak{q'}$ distinct primes with no prime ideal between them that $\mu_j(\mathfrak{q}, M_\mathfrak{p}) \neq 0$ implies that $\mu_{j+1}(\mathfrak{q'}, M_\mathfrak{p}) \neq 0$ where
\begin{displaymath}
\mu_j(\mathfrak{q}, M) = \dim_{k(\mathfrak{q})}\Ext^j_{R_\mathfrak{q}}(k(\mathfrak{q}), M_\mathfrak{q}) 
\end{displaymath}
are the Bass invariants. As $M_\mathfrak{p}$ is not perfect infinitely many of the Bass invariants are non-zero and so in particular, as $\mathfrak{p}$ is the maximal ideal of $R_\mathfrak{p}$, there are infinitely many non-zero $\mu_j(\mathfrak{p}, M_\mathfrak{p})$. Thus, taking $i$ larger if necessary, we get that
\begin{displaymath}
0 \neq \Ext^i_{R_\mathfrak{p}}( k(\mathfrak{p}), M_\mathfrak{p}) \iso \Hom_{S(R)}(I_\l Q_\r R/\mathfrak{p}, \S^i L_{\mathcal{Z}(\mathfrak{p})} I_\l Q_\r M).
\end{displaymath}
Hence $\mathit{\Gamma}_{\mathcal{V}(\mathfrak{p})}L_{\mathcal{Z}(\mathfrak{p})}I_\l Q_\r M \neq 0$ as by Lemma \ref{lem_gen1} the object $I_\l Q_\r R/\mathfrak{p}$ is one of the generators for $\mathit{\Gamma}_{\mathcal{V}(\mathfrak{p})}S(R)$. It follows that $\mathfrak{p} \in \supp I_\l Q_\r M$ as desired.
\end{proof}

\begin{prop}\label{prop_compactsupp}
Let $R$ be Gorenstein. If $a$ is a compact object of $S(R)$ then $\supp a$ is a closed subset of $\Sing R$.
\end{prop}
\begin{proof}
By the last lemma
\begin{displaymath}
\supp a = \{\mathfrak{p}\in \Sing R\; \vert \; L_{\mathcal{Z}(\mathfrak{p})} a \neq 0\}
\end{displaymath}
which is closed by Lemma \ref{lem_bigsupp}.
\end{proof}

\begin{rem}\label{rem_prop_compactsupp}
If $\mathfrak{p} \in \Sing R$ the proof of Lemma \ref{lem_compactsupp} gives the equality
\begin{displaymath}
\supp I_\l Q_\r R/\mathfrak{p} = \mathcal{V}(\mathfrak{p}).
\end{displaymath}
Indeed, as $R_\mathfrak{p}$ is not regular the residue field $k(\mathfrak{p})$ must have an infinite free resolution over $R_\mathfrak{p}$ so if $(R/\mathfrak{p})_\mathfrak{q}$ had finite projective dimension over $R$ for $\mathfrak{q}\in \mathcal{V}(\mathfrak{p})$ one could localize to find a finite resolution for $k(\mathfrak{p})$ giving a contradication. Thus
\begin{displaymath}
\supp I_\l Q_\r R/\mathfrak{p} = \{\mathfrak{q}\in \Sing R\; \vert \; L_{\mathcal{Z}(\mathfrak{q})} I_\l Q_\r R/\mathfrak{p} \neq 0\}
\end{displaymath}
which is precisely $\mathcal{V}(\mathfrak{p})$.
\end{rem}

It follows from this proposition and the compatibility of supports with extensions, coproducts, and suspensions (Proposition A.5.7) that, provided $R$ is Gorenstein, for any localizing subcategory $\mathcal{L}\cie S(R)$ generated by objects of $S(R)^c$ the subset $\s\mathcal{L} \cie \Sing R$ is specialization closed. It follows from Corollary A.4.11 that $\t$ sends specialization closed subsets to localizing subcategories of $S(R)$ generated by objects compact in $S(R)$ so $\t$ and $\s$ restrict, i.e.:

\begin{prop}\label{prop_res_assignment}
The assignments $\s$ and $\t$ restrict to well-defined functions
\begin{displaymath}
\left\{ \begin{array}{c}
\text{specialization closed} \\ \text{subsets of}\; \Sing R 
\end{array} \right\}
\xymatrix{ \ar[r]<1ex>^{\tau} \ar@{<-}[r]<-1ex>_{\sigma} &} \left\{
\begin{array}{c}
\text{localizing subcategories of}\; S(R) \\ \text{generated by objects of} \; S(R)^c
\end{array} \right\}.
\end{displaymath}
\end{prop}

We are now ready to state and prove our first classification theorem for subcategories of the singularity category (cf.\ Theorem 7.5 \cite{TakahashiMCM}).

\begin{thm}\label{thm_bijections}
Let $R$ be a commutative Gorenstein ring. Then there are order preserving bijections
\begin{displaymath}
\left\{ \begin{array}{c}
\text{subsets of}\; \Sing R 
\end{array} \right\}
\xymatrix{ \ar[r]<1ex>^{\tau} \ar@{<-}[r]<-1ex>_{\sigma} &} \left\{
\begin{array}{c}
\text{localizing subcategories}\; \mathcal{L}\; \text{of} \; S(R) \\
\text{containing} \; I_\l Q_\r k(\mathfrak{p})\; \text{for} \; \mathfrak{p} \in \s(\mathcal{L})
\end{array} \right\}
\end{displaymath}
and
\begin{displaymath}
\left\{ \begin{array}{c}
\text{specialization closed} \\ \text{subsets of}\; \Sing R 
\end{array} \right\}
\xymatrix{ \ar[r]<1ex>^{\tau} \ar@{<-}[r]<-1ex>_{\sigma} &} \left\{
\begin{array}{c}
\text{subcategories}\; \mathcal{L}\; \text{of} \; S(R) \\ 
\text{generated by objects of}\; S(R)^c\; \\ \text{and containing} \; I_\l Q_\r k(\mathfrak{p})\; \text{for} \; \mathfrak{p} \in \s(\mathcal{L})
\end{array} \right\}.
\end{displaymath}
This second being equivalent to the bijection
\begin{displaymath}
\left\{ \begin{array}{c}
\text{specialization closed} \\ \text{subsets of}\; \Sing R 
\end{array} \right\}
\xymatrix{ \ar[r]<1ex> \ar@{<-}[r]<-1ex> &} \left\{
\begin{array}{c}
\text{thick subcategories}\; \mathcal{L}\; \text{of} \; D_{\mathrm{Sg}}(R) \\
\text{such that}\; \mathcal{L}_\mathfrak{p} \cie D_{\mathrm{Sg}}(R_\mathfrak{p}) \; \text{contains}\; k(\mathfrak{p}) \\
\text{(up to summands)}
\end{array} \right\}.
\end{displaymath}
\end{thm}
\begin{proof}
We proved in Corollary \ref{cor_injective} that $\s\t(W) =W$ for every subset $W$ of $\Sing R$. If $\mathcal{L}$ is a localizing subcategory then by the local-to-global principle Theorem \ref{thm_general_ltg}
\begin{displaymath}
\mathcal{L} = \langle \mathit{\Gamma}_\mathfrak{p}\mathcal{L} \; \vert \; \mathfrak{p} \in \s(\mathcal{L})\rangle_\mathrm{loc}.
\end{displaymath}
Given that $I_\l Q_\r k(\mathfrak{p})$ lies in $\mathcal{L}$ for each $\mathfrak{p}\in \s(\mathcal{L})$ we must have $\mathit{\Gamma}_\mathfrak{p}\mathcal{L} = \mathit{\Gamma}_\mathfrak{p}S(R)$ by Lemma \ref{lem_gen2}. Thus $\mathcal{L} = \t\s(\mathcal{L})$ by Proposition \ref{prop_imtau}. 

The restricted assignments of the second claim make sense by Proposition \ref{prop_res_assignment} and it is a bijection by the same argument we have just used above.

The last bijection is a consequence of the second one together with Krause's result Theorem \ref{thm_recol} (3) which identifies $S(R)^c$, up to summands, with the singularity category $D_{\mathrm{Sg}}(R)$.
\end{proof}

\section{The classification theorem for hypersurface rings}\label{sec_hypersurface}
Throughout this section $(R,\mathfrak{m},k)$ is a local Gorenstein ring unless otherwise specified. We consider the relationship between the categories $S(R)$ and $S(R/(x))$ for $x$ a regular element. Our results allow us to classify the localizing subcategories of $S(R)$ in the case that $R$ is a hypersurface ring.

By the classification result we have already proved in Theorem \ref{thm_bijections} together with the fact that every localizing subcategory is closed under the action of $D(R)$ (Lemma \ref{lem_locpreserving}) it is sufficient to consider subcategories of $\mathit{\Gamma}_\mathfrak{p}S(R)$. Observe that a bijection between subsets of the singular locus and the collection of localizing subcategories is equivalent to each of the $\mathit{\Gamma}_\mathfrak{p}S(R)$ being minimal i.e., having no proper non-trivial localizing subcategories.

\begin{rem}
Note that we can reduce to the case of local rings when studying minimality. Indeed, suppose $R$ is a noetherian ring and $\mathfrak{p} \in \Spec R$. Then since $\mathit{\Gamma}_\mathfrak{p}R\otimes L_{\mathcal{Z}(\mathfrak{p})}R \iso \mathit{\Gamma}_{\mathfrak{p}}R$ we can study $\mathit{\Gamma}_{\mathfrak{p}}S(R)$ in $S(R_\mathfrak{p}) \cie S(R)$, the essential image of $L_{\mathcal{Z}(\mathfrak{p})}R\odot(-) \iso R_\mathfrak{p}\otimes_R(-)$. 
\end{rem}

We now prove several lemmas leading to a key proposition. The first two of these lemmas are well known so we omit the proofs.

\begin{lem}\label{lem_Gor_quot}
Let $x$ be a regular element of $R$. Then the quotient ring $R/(x)$ is also Gorenstein.
\end{lem}

\begin{lem}\label{lem_Ginj_divis}
Let $G$ be a Gorenstein injective $R$-module and $x\in R$ an $R$-regular element. Then $G$ is $x$-divisible i.e., multiplication by $x$ is surjective on $G$.
\end{lem}

\begin{notation}
We will consider $D(R)$ to act on $\underline{\GInj}R$ via the equivalence of Proposition \ref{prop_ginj_sing_equiv}. Thus by $\mathit{\Gamma}_\mathfrak{m}G$ for $G\in \underline{\GInj}R$ we mean the class represented by the Gorenstein injective $Z^0\mathit{\Gamma}_\mathfrak{m}c(G)$.
\end{notation}

\begin{lem}\label{lem_mG_ext_nonzero}
Let $G$ be a Gorenstein injective $R$-module such that $\mathit{\Gamma}_\mathfrak{m}G \neq 0$ in the stable category. Then for all $i\geq 1$ 
\begin{displaymath}
\Ext^i(k,G)\neq 0.
\end{displaymath}
\end{lem}
\begin{proof}
For $i\geq 1$ there are isomorphisms
\begin{align*}
\Ext^i(k,G) &\iso \Hom(I_\l Q_\r k, \S^i I_\l Q_\r G) \\
& \iso \Hom(I_\l Q_\r k, \S^i\mathit{\Gamma}_\mathfrak{m} I_\l Q_\r G) \\
&\iso \Ext^i(k,\mathit{\Gamma}_\mathfrak{m}G)
\end{align*}
where the first and last isomorphisms are via \cite{KrStab} Proposition 7.10, together with Lemma \ref{lem_action_compatible} for the last isomorphism, and the middle one is by adjunction and the fact that as $R$ is local there is an equality $\mathit{\Gamma}_\mathfrak{m}R = \mathit{\Gamma}_{\mathcal{V}(\mathfrak{m})}R$ of tensor idempotents in $D(R)$. By Proposition \ref{prop_subcomplex} the minimal complete resolution of $\mathit{\Gamma}_\mathfrak{m}G$ consists solely of copies of $E(k)$. Hence there is a representative for $\mathit{\Gamma}_\mathfrak{m}G$ which is $\mathfrak{m}$-torsion and, as it represents a non-zero object in the stable category, of infinite injective dimension. So using the isomorphisms above we see that the $\Ext$'s are nonvanishing as claimed: their dimensions give the cardinalities of the summands of $E(k)$ in each degree of a minimal injective resolution for our representative of $\mathit{\Gamma}_\mathfrak{m}G$ (\cite{EnochsJenda} 9.2.4).
\end{proof}

\begin{lem}\label{lem_ker_nontriv}
Let $G$ be a Gorenstein injective $R$-module such that $\mathit{\Gamma}_\mathfrak{m} G \neq 0$ in the stable category, $x\in R$ a regular element, and denote by $M$ the $R$-module $\Hom_R(R/(x), G)$. Then $\idim_R M = \infty = \pdim_R M$.
\end{lem}
\begin{proof}
As $x$ is regular we get a projective resolution of the $R$-module $R/(x)$
\begin{displaymath}
\xymatrix{
0 \ar[r] & R \ar[r]^x & R \ar[r] & R/(x) \ar[r] & 0.
}
\end{displaymath}
Recall from Proposition \cite{EnochsJenda} Corollary 11.2.2 that Gorenstein injective modules are right $\Ext^i$-orthogonal to the modules of finite projective dimension for $i\geq 1$. So applying $\Hom_R(-,G)$ to the above short exact sequence yields an exact sequence
\begin{displaymath}
\xymatrix{
0 \ar[r] & \Hom_R(R/(x), G) \ar[r] & G \ar[r]^x & G \ar[r] & 0.
}
\end{displaymath}
Applying $\Hom_R(k,-)$ gives a long exact sequence
\begin{displaymath}
\xymatrix{
0 \ar[r] & \Hom(k,M) \ar[r] & \Hom(k,G) \ar[r] & \Hom(k,G) \ar[r] & \Ext^1(k,M) \ar[r] & \cdots
}
\end{displaymath}
where for $i\geq 0$ the maps
\begin{displaymath}
\Ext^i(k,G) \to \Ext^i(k,G)
\end{displaymath}
are multiplication by $x$ (see for example \cite{Weibel:HA} 3.3.6) and hence are all $0$ as (\cite{Weibel:HA} Corollary 3.3.7) the $\Ext^i(k,G)$ are $k$-vector spaces. Thus for $i\geq 0$ the morphisms
\begin{displaymath}
\Ext^i(k,M) \to \Ext^{i}(k,G)
\end{displaymath}
are surjective. By the last lemma the groups $\Ext^i(k,G)$ are non-zero for $i\geq 1$. Thus $\Ext^i(k,M)\neq 0$ for $i\geq 1$ so $M$ necessarily has infinite injective dimension. Since $R$ is Gorenstein $M$ must also have infinite projective dimension (see for example \cite{EnochsJenda} Theorem 9.1.10).
\end{proof}

\begin{prop}\label{prop_inductive}
Let $G$ be a non-zero object of $\mathit{\Gamma}_\mathfrak{m}\underline{\GInj}R$ and suppose $x$ is a regular element of $R$. Then $\langle G \rangle_{\mathrm{loc}}$ contains the image of a non-injective Gorenstein injective envelope of an $R/(x)$-module.
\end{prop}
\begin{proof}
By Lemmas \ref{lem_ker_nontriv} and \ref{lem_Ginj_divis} there is a short exact sequence of $R$-modules
\begin{displaymath}
0 \to M \to G \stackrel{x}{\to} G \to 0
\end{displaymath}
with $M$ an $R/(x)$-module of infinite projective dimension over $R$. Applying $Z^0I_\l Q_\r$ gives a triangle in $\langle G \rangle_{\mathrm{loc}}$
\begin{displaymath}
G_R(M) \to G \to G \to \S G_R(M)
\end{displaymath}
where we use Corollary \ref{cor_envelope} to identify $Z^0I_\l Q_\r M$ with the class of its Gorenstein injective envelope. The module $G_R(M)$ is not injective by Theorem \ref{thm_env_exist} as $\pdim_R M = \infty$ (i.e., it is not in the kernel of $I_\l Q_\r$), which completes the proof.
\end{proof}

Suppose $R$ is an artinian local hypersurface. Then necessarily $R$ is, up to isomorphism, of the form $S/(x^n)$ for a discrete valuation ring $S$ with $x$ a uniformiser. In particular, $R$ is an artinian principal ideal ring so by \cite{Huisgen} Theorem 2 every $R$-module is a direct sum of cyclic $R$-modules. Using this fact we show that $\underline{\GInj}R$ is minimal when $R$ is an artinian local hypersurface. This provides the base case for our inductive argument that the maps of Theorem \ref{thm_bijections} are bijections for any hypersurface ring without the requirement that the categories in question contain certain objects.

\begin{lem}\label{lem_art_hyper_min}
Suppose $R$ is an artinian local hypersurface. Then the category $\underline{\GInj}R = \mathit{\Gamma}_\mathfrak{m}\underline{\GInj}R$ is minimal.
\end{lem}
\begin{proof}
Since $R$ is 0-Gorenstein there is an equality $\underline{R\text{-}\Module} = \underline{\GInj}R$, where $\underline{R\text{-}\Module}$ is the stable category of the Frobenius category $R$-$\Module$, as every $R$-module is Gorenstein injective by \cite{EnochsJenda} Proposition 11.2.5 (4).

As remarked above we have an isomorphism $R\iso S/(x^n)$ where $S$ is a discrete valuation ring and $x$ is a uniformiser. We also recalled that every $R$-module is a coproduct of cyclic $R$-modules; this remains true in the stable category $\underline{R\text{-}\Module}$. Since the subcategory of compacts in $\underline{R\text{-}\Module}$ is precisely $\underline{R\text{-}\module}$ every object is thus a coproduct of compact objects (so in particular $\underline{R\text{-}\Module}$ is a pure-semisimple triangulated category cf.\ \cite{BelPurity} Corollary 12.26). It follows that every localizing subcategory of $\underline{R\text{-}\Module}$ is generated by objects of $\underline{R\text{-}\module}$.

We deduce minimality from the existence of Auslander-Reiten sequences. For each $1\leq i \leq n-1$ there is an Auslander-Reiten sequence
\begin{displaymath}
\xymatrix{
0 \ar[r] & R/(x^i) \ar[r]^(0.3)f & R/(x^{i-1}) \oplus R/(x^{i+1}) \ar[r]^(0.7)g & R/(x^i) \ar[r] & 0
}
\end{displaymath}
where, using $\overline{(-)}$ to denote residue classes,
\begin{displaymath}
f(\overline{a}) =  \begin{pmatrix}\overline{a} \\ \overline{ax} \end{pmatrix} \quad \text{and} \quad g \begin{pmatrix} \overline{a} \\ \overline{b}\end{pmatrix} = \overline{ax - b}.
\end{displaymath}
So the smallest thick subcategory containing any non-zero compact object is all of $\underline{R\text{-}\module}$: every object is up to isomorphism a coproduct of the classes of cyclic modules and the Auslander-Reiten sequences show that any cyclic module which is non-zero in the stable category is a generator. Thus, as we observed above that every localizing subcategory is generated by objects of $\underline{R\text{-}\module}$, we see there are no non-trivial localizing subcategories except for all of $\underline{R\text{-}\Module}$ so it is minimal as claimed.
\end{proof}

We next need a result that is essentially contained in \cite{KrStab} Section 6 but which we reformulate in a way which is more convenient for our purposes. To prove this lemma we need to recall the following well known fact.

\begin{lem}\label{lem_lol_soeasy}
Suppose $R$ and $S$ are local rings and $\pi\colon R\to S$ is a surjection with kernel generated by an $R$-regular sequence. Then the functor
\begin{displaymath}
\pi_* \colon S\text{-}\Module \to R \text{-}\Module
\end{displaymath}
sends modules of finite projective dimension to modules of finite projective dimension.
\end{lem}


\begin{lem}\label{lem_stable_map}
Suppose $\pi\colon R \to S$ is a surjective map of Gorenstein local rings with kernel generated by an $R$-regular sequence. Then there is an induced coproduct preserving exact functor
\begin{displaymath}
\underline{\pi_*}\colon \underline{\GInj}S \to \underline{\GInj}R
\end{displaymath}
which sends an object of $\underline{\GInj}S$ to its $\GInj R$-envelope.
\end{lem}
\begin{proof}
Let us denote by $\nu$ the composite
\begin{displaymath}
\nu\colon \xymatrix{ D(S) \ar[r]^{\pi_*} & D(R) \ar[r]^{I_\l Q_\r} & S(R) \ar[r]^{Z^0} & \underline{\GInj} R}.
\end{displaymath}
Recall from \cite{KrStab} Corollary 5.5 and Example 5.6 that the composite
\begin{displaymath}
\mu\colon \xymatrix{ R\text{-}\Module \ar[r] & D(R) \ar[r]^{I_\l Q_\r} & S(R)}
\end{displaymath}
where the functor $R$-$\Module \to D(R)$ is the canonical inclusion, preserves all coproducts and annihilates all modules of finite projective dimension. Thus by Lemma \ref{lem_lol_soeasy} the equal composites
\begin{displaymath}
\xymatrix{
\GInj S \ar[r] \ar[dr]_{\pi_*} & D(S) \ar[r]^\nu & \underline{\GInj}R \\
& R\text{-}\Module \ar[ur]_{Z^0\mu} &
}
\end{displaymath}
must factor via the stable category $\underline{\GInj} S$. Indeed, as $R$ is Gorenstein $S$ is also Gorenstein by Lemma \ref{lem_Gor_quot}. Thus injective $S$-modules have finite $S$-projective dimension and $\pi_*$ sends them to modules of finite $R$-projective dimension by the last lemma. In particular $S$-injectives are killed by both composites. We get a commutative diagram
\begin{displaymath}
\xymatrix{
\GInj S \ar[r] \ar[dr]_p & D(S) \ar[r]^\nu & \underline{\GInj}R \\
& \underline{\GInj} S \ar[ur]_{\underline{\pi_*}}. &
}
\end{displaymath}
The functors $\pi_*$, $p$, and $Z^0\mu$ are all coproduct preserving: we have already noted that $\mu$ preserves coproducts, $Z^0$ is the equivalence of Proposition \ref{prop_ginj_sing_equiv}, and it is easily seen that the projection $p$ also preserves coproducts (the concerned reader may consult \cite{KrStab} Corollary 7.14).  As $p$ is essentially surjective we see that $\underline{\pi_*}$ also preserves coproducts. Indeed, the top composite is equal to $Z^0\mu \pi_*$ which preserves coproducts and any coproduct of objects in $\underline{\GInj}S$ is the image under $p$ of a coproduct of $S$-modules. Exactness follows similarly by noting that the top composite sends short exact sequences to triangles as $\GInj S$ is an exact subcategory of $S$-$\Module$ and $\pi_*$ is exact.

The explicit description of $\underline{\pi_*}$ is clear from the construction: by the commutativity of the diagram $\underline{\pi_*}$ sends the image of an object $M$ of $\GInj S$ to $Z^0I_\l Q_\r \pi_* M$ which is precisely its Gorenstein injective envelope as an $R$-module by Corollary \ref{cor_envelope}.
\end{proof}

\begin{rem}
Given an $S$-module $M$ we see from the above that, letting $G_S(M)$ and $G_R(\pi_* M)$ denote its Gorenstein injective envelopes over $S$ and $R$ respectively, there are isomorphisms in the stable category
\begin{displaymath}
\underline{\pi_*}G_S(M) \iso G_R(\pi_* G_S(M)) \iso G_R(\pi_* M).
\end{displaymath}
The first isomorphism is a consequence of the last lemma. The second isomorphism follows from Theorem \ref{thm_env_exist} which provides us, after an application of $\pi_*$, with a short exact sequence
\begin{displaymath}
0 \to \pi_* M \to \pi_* G_S(M) \to \pi_* L \to 0
\end{displaymath}
where $\pi_*L$ has finite projective and injective dimension. Thus the $R$-Gorenstein injective envelopes of $\pi_* M$ and $\pi_*G_S(M)$ agree in $\underline{\GInj} R$ which gives the second isomorphism.
\end{rem}


We are now ready to prove the theorem which gives us a complete classification of the localizing subcategories of $S(R)$ when $R$ is a hypersurface.

\begin{thm}\label{thm_hyper_min}
If $(R,\mathfrak{m},k)$ is a hypersurface then $\mathit{\Gamma}_\mathfrak{m}\underline{\GInj}R$ is minimal.
\end{thm}
\begin{proof}
We prove the theorem by induction on the dimension of $R$. In the case $\dim R =0$ then $R$ is an artinian hypersurface and $\underline{\GInj}R$ is minimal by Lemma \ref{lem_art_hyper_min}.

So suppose the theorem holds for hypersurfaces of dimension strictly less than $n$ and let $\dim R = n\geq 1$. Then as $\depth R = n \geq 1$ the maximal ideal $\mathfrak{m}$ is not contained in any of the associated primes or $\mathfrak{m}^2$ so we can choose, by prime avoidance, a regular element $x$ not lying in $\mathfrak{m}^2$. The quotient $R/(x)$ is again a hypersurface, for example one can see this by noting that the second deviations agree $\varepsilon_2(R) = \varepsilon_2(R/(x)) = 1$ and the higher deviations vanish (see \cite{AvramovIFR} section 7 for details).


Let us denote the projection $R\to R/(x)$ by $\pi$. By Proposition \ref{prop_inductive} for every $0\neq G \in \mathit{\Gamma}_\mathfrak{m}\underline{\GInj} R$ the subcategory $\langle G \rangle_{\mathrm{loc}}$ contains a non-zero object in the image of the functor $\underline{\pi_*}$ of Lemma \ref{lem_stable_map}. The ring $R/(x)$ has dimension $n-1$ so by the inductive hypothesis the category $\mathit{\Gamma}_{\mathfrak{m}/(x)}\underline{\GInj}R/(x)$ is minimal.

The functor $\underline{\pi_*}$ is exact and coproduct preserving by Lemma \ref{lem_stable_map} so as $\langle G \rangle_{\mathrm{loc}}$ contains one object in the image of $\underline{\pi_*}$ it must contain the whole image by minimality of $\mathit{\Gamma}_{\mathfrak{m}/(x)}\underline{\GInj}R/(x)$. In particular $\langle G \rangle_\mathrm{loc}$ contains $G_R(k) \iso \underline{\pi_*}G_{R/(x)}(k)$. This object generates $\mathit{\Gamma}_\mathfrak{m}\underline{\GInj}R$ by $Z^0$ applied to the statement of Lemma \ref{lem_gen2}. Hence $\langle G \rangle_{\mathrm{loc}} = \mathit{\Gamma}_\mathfrak{m}\underline{\GInj}R$ so $\mathit{\Gamma}_\mathfrak{m}\underline{\GInj}R$ is minimal as claimed.
\end{proof}

Using the other techniques we have developed this is enough to give a classification of the localizing subcategories of $S(R)$ for $R$ a not necessarily local ring which is locally a hypersurface.

\begin{thm}\label{cor_hyper_min}
If $R$ is a noetherian ring which is locally a hypersurface then there is an order preserving bijection
\begin{displaymath}
\left\{ \begin{array}{c}
\text{subsets of}\; \Sing R 
\end{array} \right\}
\xymatrix{ \ar[r]<1ex>^{\tau} \ar@{<-}[r]<-1ex>_{\sigma} &} \left\{
\begin{array}{c}
\text{localizing subcategories of} \; S(R) \\
\end{array} \right\} 
\end{displaymath}
given by the assignments of Theorem \ref{thm_bijections}. This restricts to the equivalent order preserving bijections
\begin{displaymath}
\left\{ \begin{array}{c}
\text{specialization closed} \\ \text{subsets of}\; \Sing R 
\end{array} \right\}
\xymatrix{ \ar[r]<1ex>^{\tau} \ar@{<-}[r]<-1ex>_{\sigma} &} \left\{
\begin{array}{c}
\text{localizing subcategories of} \; S(R)\; \\
\text{generated by objects of} \; S(R)^c
\end{array} \right\} 
\end{displaymath}
and
\begin{displaymath}
\left\{ \begin{array}{c}
\text{specialization closed} \\ \text{subsets of}\; \Sing R 
\end{array} \right\}
\xymatrix{ \ar[r]<1ex> \ar@{<-}[r]<-1ex> &} \left\{
\begin{array}{c}
\text{thick subcategories of} \; D_{\mathrm{Sg}}(R) \\
\end{array} \right\}.
\end{displaymath}
\end{thm}
\begin{proof}
By Theorem \ref{thm_bijections} it is sufficient to check every localizing subcategory $\mathcal{L}$ contains $I_\l Q_\r k(\mathfrak{p})$ for each $\mathfrak{p}\in \Sing R$ such that $\mathit{\Gamma}_\mathfrak{p}\mathcal{L} \neq 0$. As there are equivalences
\begin{displaymath}
\mathit{\Gamma}_\mathfrak{p}S(R) \iso \mathit{\Gamma}_\mathfrak{p}S(R_\mathfrak{p})
\end{displaymath}
each of the subcategories $\mathit{\Gamma}_\mathfrak{p}S(R)$ is minimal by Theorem \ref{thm_hyper_min} as each $R_\mathfrak{p}$ is a local hypersurface. Hence if $\mathit{\Gamma}_\mathfrak{p}\mathcal{L} \neq 0$ for a localizing subcategory $\mathcal{L}$ we must have $\mathit{\Gamma}_\mathfrak{p}S(R) = \mathit{\Gamma}_\mathfrak{p}\mathcal{L} \cie \mathcal{L}$ where the containment is a consequence of the closure of localizing subcategories under the action of $D(R)$ (Lemma \ref{lem_locpreserving}). In particular the generator $I_\l Q_\r k(\mathfrak{p})$ of $\mathit{\Gamma}_\mathfrak{p}S(R)$ is an object of $\mathcal{L}$. Thus the image of the injection $\t$, namely those localizing subcategories $\mathcal{L}$ containing $I_\l Q_\r k(\mathfrak{p})$ for each $\mathfrak{p}\in \Sing R$ such that $\mathit{\Gamma}_\mathfrak{p}\mathcal{L} \neq 0$, is in fact the set of all localizing subcategories. This proves the first bijection.

As in Theorem \ref{thm_bijections} the second bijection is a consequence of the first and Proposition \ref{prop_compactsupp} which states that compact objects have closed supports so $\s$ of a compactly generated subcategory is specialization closed. The third bijection is equivalent to the second as by Theorem \ref{thm_recol} (3) there is an equivalence up to summands $D_{\mathrm{Sg}}(R) \iso S(R)^c$ so our restatement is a consequence of Thomason's localization theorem (\cite{NeeGrot} Theorem 2.1).
\end{proof}

\begin{rem}
Our result implies Takahashi's Theorem 7.6 of \cite{TakahashiMCM}.
\end{rem}

We can use this theorem to give a proof of the telescope conjecture for $S(R)$ when $R$ is locally a hypersurface.
\begin{thm}\label{thm_hyper_tele}
If $R$ is locally a hypersurface then the singularity category $S(R)$ satisfies the telescope conjecture i.e., every smashing subcategory of $S(R)$ is generated by objects of $S(R)^c$.
\end{thm}
\begin{proof}
As every localizing subcategory of $S(R)$ is a $D(R)$-submodule by Lemma \ref{lem_locpreserving} and the $D(R)$ action classifies the localizing subcategories of $S(R)$ by Theorem \ref{cor_hyper_min} the relative telescope conjecture (Definition \ref{defn_relative_tele}) for this action agrees with the usual telescope conjecture. Thus it is sufficient to verify that the conditions of Theorem \ref{thm_rel_tele} hold.

The local-to-global principle holds for the action as Theorem \ref{thm_general_ltg} applies to $D(R)$. The support of every compact object of $S(R)$ is specialization closed by Proposition \ref{prop_compactsupp} and for every irreducible closed subset $\mathcal{V}(\mathfrak{p}) \cie \Sing R$ the object $I_\l Q_\r R/\mathfrak{p}$ has support $\mathcal{V}(\mathfrak{p})$ by Remark \ref{rem_prop_compactsupp}.


Thus the theorem applies and every smashing subcategory of $S(R)$ is compactly generated.
\end{proof}

\section{Schemes and hypersurface singularities}\label{sec_schemes}
We are now in a position to demonstrate that what we have proved in the affine case extends in a straightforward way to noetherian separated schemes via the machinery of Section 8 of \cite{StevensonActions}. As in Definition \ref{defn_vis_sigmatau} we have assignments
\begin{displaymath}
\left\{ \begin{array}{c}
\text{subsets of}\; X 
\end{array} \right\}
\xymatrix{ \ar[r]<1ex>^\t \ar@{<-}[r]<-1ex>_\s &} \left\{
\begin{array}{c}
\text{localizing $D(X)$-submodules} \; \text{of} \; S(X) \\
\end{array} \right\} 
\end{displaymath}
where for a localizing submodule $\mathcal{L}$ we set
\begin{displaymath}
\s(\mathcal{L}) = \supp \mathcal{L} = \{x \in X \; \vert \; \mathit{\Gamma}_x\mathcal{L} \neq 0\}
\end{displaymath}
and for a subset $W$ of $X$
\begin{displaymath}
\t(W) = \{A \in S(X) \; \vert \; \supp A \cie W\}.
\end{displaymath}

In this section, unless stated otherwise, submodules are localizing submodules. In order to apply our formalism to the situation of $D(X)$ acting on $S(X)$ we first need to understand what the effect of restricting to an open subset of $X$ is.

Before doing this let us remind the reader of some of notation. Given a specialization closed subset $\mathcal{V} \cie X$ we denote by $D_\mathcal{V}(X)$ the smashing subcategory generated by those compact objects whose support, in the sense of \cite{BaSpec}, lies in $\mathcal{V}$. We recall that the corresponding localization sequence gives rise to the tensor idempotents $\mathit{\Gamma}_\mathcal{V}\str_X$ and $L_\mathcal{V}\str_X$. For a closed subset $Z$ of $X$ with complement $U$ we denote the quotient $D(X)/D_Z(X)$ by either $L_ZD(X)$ or $D(X)(U)$, as in Section A.8. The action of $D(X)$ on $S(X)$ gives rise to an action of $D(X)(U)$ on $S(X)(U) = L_ZS(X)$ as in Proposition A.8.5.

\begin{lem}\label{lem_open_loc_a}
Let $U\cie X$ be an open set with complement $Z = X\setminus U$, and let $f\colon U\to X$ be the inclusion. If $E$ is an object of $D(X)$ then the map \mbox{$E\to \mathbf{R}f_*f^*E$} agrees with the localization map $E\to L_ZE$. In particular, $D(X)(U)$ is precisely $D(U)$.
\end{lem}
\begin{proof}
By definition the smashing subcategory $D_Z(X)$ giving rise to $L_Z$ is the localizing subcategory generated by the compact objects whose support is contained in $Z$. The kernel of $f^*$ is the localizing subcategory generated by those compact objects whose homological support is contained in $Z$ by \cite{Rouquier_dim}. As these two notions of support coincide for compact objects of $D(X)$ (see for example \cite{BaSpec} Corollary 5.6) the lemma follows immediately.
\end{proof}

We recall from \cite{KrStab} Theorems 1.5 and 6.6 that for $f\colon U\to X$ an open immersion we obtain an adjoint pair of functors
\begin{displaymath}
\xymatrix{
S(X) \ar[r]<0.5ex>^{f^*} \ar@{<-}[r]<-0.5ex>_{f_*} & S(U).
}
\end{displaymath}
These functors are easily seen to be, using the classification of injective quasi-coherent sheaves on a locally noetherian scheme (see for example \cite{ConradDuality} Lemma 2.1.5), just the usual pullback and pushforward of complexes.

\begin{lem}\label{lem_open_loc_b}
With notation as in Lemma \ref{lem_open_loc_a} suppose $U\cie X$ is an open affine and let $A$ be an object of $S(X)$. Then the natural map $A\to f_*f^*A$ agrees with $A\to L_Z A$. In particular, $S(X)(U)$ is identified with $S(U)$.
\end{lem}
\begin{proof}
Since $f\colon U\to X$ is an affine morphism we have that $f_*\colon D(U)\to D(X)$ is exact and $\mathbf{R}f_* = f_*$. The map $A\to L_ZA$ is, by definition, obtained by taking the morphism $\str_X \to L_Z\str_X$ in $D(X)$ and tensoring with $A\in S(X)$. By Lemma \ref{lem_open_loc_a} the map $\str_X \to L_Z\str_X$ is just $\str_X \to f_*\str_U$, which is a map of K-flat complexes. Thus the map $A \to L_Z A$ is
\begin{displaymath}
A \to f_*\str_U \otimes_{\str_X} A \iso f_*(\str_U \otimes_{\str_U} f^*A) \iso f_*f^*A,
\end{displaymath}
where the first isomorphism is by the projection formula, completing the proof.
\end{proof}

Now we are in business: we know that for an open affine $U\iso \Spec R$ in $X$ the construction of Section A.8 gives us $D(R)$ acting on $S(R)$. It just remains to verify that this is the action we expect.

\begin{lem}\label{lem_aff_reduction}
Suppose $U$ is an open subscheme of $X$ with inclusion $f\colon U\to X$. Then the diagram
\begin{displaymath}
\xymatrix{
D(X) \times S(X) \ar[d]_{\odot} \ar[rr]^{f^*\times f^*} && D(U) \times S(U) \ar[d]^{\odot} \\
S(X) \ar[rr]^{f^*} && S(U)
}
\end{displaymath}
commutes up to natural isomorphism.
\end{lem}
\begin{proof}
By virtue of being an open immersion $f^*$ sends K-flat complexes to K-flat complexes and commutes with taking K-flat resolutions. Thus, as $f^*$ commutes with tensor products up to natural isomorphism, resolving by a K-flat, tensoring, and then pulling back agrees with pulling back, resolving and then tensoring (up to natural isomorphism). So the square is commutative as claimed.
\end{proof}

This is the diagram of Proposition A.8.5, so it follows that the action $\odot_U$ of said proposition is precisely our old friend $\odot$. Thus we can use the machinery we have developed to obtain a classification of the localizing $D(X)$-submodules of $S(X)$ when $X$ is locally a hypersurface.

\begin{lem}\label{lem_nonaffine_nontriv}
There is an equality $\s S(X) = \Sing X$ i.e., for any $x\in X$ the subcategory $\mathit{\Gamma}_xS(X)$ is non-trivial if and only if $x\in \Sing X$.
\end{lem}
\begin{proof}
Let $\Un_{i=1}^n U_i$ be an open affine cover of $X$. By Remark A.8.7 the subset $\s S(X)$ is the union of the $\s S(U_i)$. Thus it is sufficient to note that $x\in U_i$ lies in $\Sing X$ if and only if it lies in $\Sing U_i$ and invoke Proposition \ref{prop_nontriv} which tells us that $\s S(U_i) = \Sing U_i$.
\end{proof}

\begin{prop}\label{prop_nonaffine_specclosed}
If $X$ is a Gorenstein separated scheme then every compact object of $S(X)$ has closed support.
\end{prop}
\begin{proof}
We proved that for any open affine $U_i$ the compact objects of $S(U_i)$ have closed support in Proposition \ref{prop_compactsupp}. The result then follows by covering $X$ by open affines and applying Lemma A.8.8.
\end{proof}

\begin{rem}
It follows that the support of any submodule generated by compact objects of $S(X)$ is a specialization closed subset of $\Sing X$.
\end{rem}

We are now ready to state our first theorem concerning the singularity categories of schemes with hypersurface singularities.

\begin{thm}\label{thm_hyper_nonaffine_bijections}
Suppose $X$ is a noetherian separated scheme with only hypersurface singularities. Then there is an order preserving bijection
\begin{displaymath}
\left\{ \begin{array}{c}
\text{subsets of}\; \Sing X 
\end{array} \right\}
\xymatrix{ \ar[r]<1ex>^{\tau} \ar@{<-}[r]<-1ex>_{\sigma} &} \left\{
\begin{array}{c}
\text{localizing submodules of} \; S(X) \\
\end{array} \right\} 
\end{displaymath}
given by the assignments discussed before Lemma \ref{lem_aff_reduction}. This restricts to the equivalent  bijections
\begin{displaymath}
\left\{ \begin{array}{c}
\text{specialization closed} \\ \text{subsets of}\; \Sing X 
\end{array} \right\}
\xymatrix{ \ar[r]<1ex>^{\tau} \ar@{<-}[r]<-1ex>_{\sigma} &} \left\{
\begin{array}{c}
\text{submodules of} \; S(X)\; \text{generated} \\
\text{by objects of} \; S(X)^c
\end{array} \right\} 
\end{displaymath}
and
\begin{displaymath}
\left\{ \begin{array}{c}
\text{specialization closed} \\ \text{subsets of}\; \Sing X 
\end{array} \right\}
\xymatrix{ \ar[r]<1ex> \ar@{<-}[r]<-1ex> &} \left\{
\begin{array}{c}
\text{thick} \; D^{\mathrm{perf}}(X)\text{-submodules of} \; D_{\mathrm{Sg}}(X) \\
\end{array} \right\}.
\end{displaymath}
\end{thm}
\begin{proof}
The first bijection is an application of Theorem \ref{cor_hyper_min} and Theorem A.8.11 to an open affine cover of $X$ together with the observation of Lemma \ref{lem_nonaffine_nontriv} that $\s S(X) = \Sing X$.
To see that the first bijection restricts to the second recall from Proposition \ref{prop_nonaffine_specclosed} that compact objects of $S(X)$ have specialization closed support. The statement now follows immediately from what we have already proved and using Theorem \ref{thm_recol} it is easily deduced that the second and third bijections are equivalent.
\end{proof}

It is natural to ask when one can strengthen this result to a complete classification of the localizing subcategories of $S(X)$. We now prove that if $X$ is a hyperplane section of a regular scheme then every localizing subcategory of $S(X)$ is closed under the action of $D(X)$. This gives a complete description of the lattice of localizing subcategories of $S(X)$ for such schemes. 

Let $T$ be a regular separated noetherian scheme of finite Krull dimension and let $L$ be an ample line bundle on $T$. Suppose $t\in H^0(T,L)$ is a section giving rise to an exact sequence
\begin{displaymath}
0 \to L^{-1} \stackrel{t^\vee}{\to} \str_T \to \str_X \to 0
\end{displaymath}
which defines a hypersurface $X \stackrel{i}{\to} T$. The scheme $X$ is a noetherian separated scheme with hypersurface singularities so our theorem applies to classify localizing $D(X)$-submodules of $S(X)$. The key observation in strengthening this result is the following easy computation.

\begin{lem}\label{lem_lolcool}
Let $F\in D(X)$ be a quasi-coherent sheaf concentrated in degree zero. There is an isomorphism in $S(X)$
\begin{displaymath}
I_\l Q_\r (F\otimes i^*L^{-1}) \iso \S^{-2} I_\l Q_\r F.
\end{displaymath}
\end{lem}
\begin{proof}
By the way we have defined $X$ the coherent $\str_T$-module $\str_X$ comes with a flat resolution
\begin{displaymath}
0\to L^{-1} \stackrel{t^\vee}{\to} \str_{T} \to \str_X \to 0.
\end{displaymath}
Thus the complex $\mathbf{L}i^*i_*F$ has two non-zero cohomology groups namely
\begin{displaymath}
H^0(\mathbf{L}i^*i_*F) \iso F \quad \text{and} \quad H^{-1}(\mathbf{L}i^*i_*F) \iso F\otimes_{\str_X} i^*L^{-1}.
\end{displaymath}
As the scheme $T$ is regular of finite Krull dimension the object $i_*F$ of $D(T)$ is locally isomorphic to a bounded complex of projectives. Hence $\mathbf{L}i^*i_*F$ is also locally isomorphic to a bounded complex of projectives. In particular, since being the zero object is local in $S(X)$ by Lemma \ref{lem_open_loc_b} and the local-to-global principle, we have $I_\l Q_\r \mathbf{L}i^*i_*F \iso 0$. The standard t-structure on $D(X)$ gives a triangle
\begin{displaymath}
\S F\otimes_{\str_X} i^*L^{-1} \to \mathbf{L}i^*i_*F \to F \to \S^2 F\otimes_{\str_X} i^*L^{-1}.
\end{displaymath}
Thus applying $I_\l Q_\r$ to this triangle yields an isomorphism 
\begin{displaymath}
I_\l Q_\r F \iso I_\l Q_\r (\S^2 F\otimes_{\str_X} i^*L^{-1})
\end{displaymath}
in $S(X)$ i.e., $I_\l Q_\r (F\otimes_{\str_X} i^*L^{-1}) \iso \S^{-2} I_\l Q_\r F$.
\end{proof}

Let us write $i^*L^n$ for the tensor product of $n$ copies of $i^*L$. By Proposition \ref{prop_action_compatible} and Lemma \ref{lem_action_compatible} twisting by $i^*L^{n}$ and applying $I_\l Q_\r$ to a sheaf $F$ commute up to natural isomorphism. We thus have isomorphisms
\begin{displaymath}
i^*L^n \odot I_\l Q_\r F \iso I_\l Q_\r (F\otimes_{\str_X} i^*L^{n}) \iso \S^{2n} I_\l Q_\r F
\end{displaymath}
in $S(X)$.

\begin{cor}\label{cor_ample_stable}
Let $X$ be as above. Then there are order preserving bijections
\begin{displaymath}
\left\{ \begin{array}{c}
\text{subsets of}\; \Sing X 
\end{array} \right\}
\xymatrix{ \ar[r]<1ex>^{\tau} \ar@{<-}[r]<-1ex>_{\sigma} &} \left\{
\begin{array}{c}
\text{localizing subcategories of} \; S(X) \\
\end{array} \right\} 
\end{displaymath}
and
\begin{displaymath}
\left\{ \begin{array}{c}
\text{specialization closed} \\ \text{subsets of}\; \Sing X 
\end{array} \right\}
\xymatrix{ \ar[r]<1ex>^{\tau} \ar@{<-}[r]<-1ex>_{\sigma} &} \left\{
\begin{array}{c}
\text{localizing subcategories} \\ \text{of} \; S(X) \; \text{generated by} \\ \text{objects of} \; S(X)^c
\end{array} \right\}. 
\end{displaymath}
\end{cor}
\begin{proof}
As $X$ is a locally complete intersection in the regular scheme $T$ it is certainly Gorenstein. In particular it has a dualising complex so by \cite{MurfetTAC} (Proposition 6.1 and Theorem 4.31) every complex in $S(X)$ is totally acyclic. Thus \cite{KrStab} Proposition 7.13 applies telling us that every object of $S(X)$ is the image, under $I_\l Q_\r$, of a Gorenstein injective sheaf on $X$.

Let $\mathcal{L}\cie S(X)$ be a localizing subcategory and suppose $A$ is an object of $\mathcal{L}$. Then there exists a Gorenstein injective sheaf $G$ such that $A\iso I_\l Q_\r G$ by the discussion above. There are isomorphisms
\begin{align*}
\S^m i^*L^n \odot A \iso \S^m i^*L^n\odot I_\l Q_\r G &\iso \S^m I_\l Q_\r  (G \otimes i^*L^n) \\
&\iso \S^{m+2n}I_\l Q_\r G \\
&\iso \S^{m+2n} A
\end{align*}
where we can interchange the action of $i^*L^n$ and $I_\l Q_\r$ as in the discussion before the proposition.

As $L$ is ample on $T$ the line bundle $i^*L$ is ample on $X$ so the set of objects
\begin{displaymath}
\{ \S^m i^*L^n \; \vert \; m,n\in \int\}
\end{displaymath}
is a compact generating set for $D(X)$, see for example 1.10 of \cite{NeeGrot}. We have just seen $\mathcal{L}$ is stable under the action of each of the generators. Thus the full subcategory of $D(X)$ consisting of objects whose action sends $\mathcal{L}$ to itself is localizing, as $\mathcal{L}$ is localizing, and contains a generating set so must be all of $D(X)$. This proves $\mathcal{L}$ is a submodule as claimed.
\end{proof}

\begin{rem}
The action of $i^*L$ can be viewed in the context of the degree 2 periodicity operator of Gulliksen \cite{GulliksenOperator} (see also \cite{EisenbudOperators} and \cite{AvSunOperators}). As $i^*L$ is invertible in $D(X)$ one can consider, as in \cite{BaSSS}, the graded commutative ring
\begin{displaymath}
E_{i^*L}^* = \bigoplus_{j\in \int} \Hom(\str_X, i^*L^j)
\end{displaymath}
with multiplication defined by sending $(\str_X \to i^*L^j, \str_X \to i^*L^k)$ to the composite
\begin{displaymath}
\xymatrix{
\str_X \ar[d] \ar[rr] & & i^*L^{j+k} \\
i^*L^j \ar[r]^(0.4)\sim & i^*L^j \otimes \str_X \ar[r] & i^*L^j\otimes i^*L^k. \ar[u]^\wr
}
\end{displaymath}
The degree $j$ elements of the ring $E_{i^*L}^*$ act on $S(X)$ by natural transformations $\id_{S(X)} \to i^*L^j \otimes(-)$. In particular, in the above situation Lemma \ref{lem_lolcool} implies that $E_{i^*L}^*$ acts via the even part of the central ring $\mathcal{Z}(S(X))$.
\end{rem}

To end the section we show that our relative version of the telescope conjecture (Definition \ref{defn_relative_tele}) holds for the action of $D(X)$ on $S(X)$ when $X$ is any separated noetherian scheme with hypersurface singularities.

\begin{lem}\label{lem_closed_compact_exists}
Let $X$ be a Gorenstein separated scheme. For any irreducible closed subset $\mathcal{V} \cie \Sing X$ there exists a compact object of $S(X)^c$ whose support is precisely $\mathcal{V}$, namely $I_\l Q_\r \str_\mathcal{V}$ where $\str_\mathcal{V}$ is the structure sheaf associated to the reduced induced structure on $\mathcal{V}$.
\end{lem}
\begin{proof}
Let $\mathcal{V}$ be an irreducible closed subset of $\Sing X$ as in the statement. We have claimed the object $I_\l Q_\r \str_\mathcal{V}$ of $S(X)^c$ has the desired support. To see this let $X$ be covered by open affine subschemes $\{U_i\}_{i=1}^n$ where $U_i \iso \Spec R_i$. The restriction $\str_{\mathcal{V}_i}$ of $\str_\mathcal{V}$ to $U_i$ is the sheaf associated to \mbox{$R/\mathfrak{p}_i$ where $\mathcal{V}(\mathfrak{p}_i) = \mathcal{V}_i = \mathcal{V}\intersec U_i$}. By Remark A.8.9
\begin{align*}
\supp I_\l Q_\r \str_\mathcal{V} &= \Un_{i=1}^n \supp I_\l Q_\r \str_{\mathcal{V}_i} \\
&= \Un_{i=1}^n \supp I_\l Q_\r R/\mathfrak{p}_i \\
&= \Un_{i=1}^n \mathcal{V}_i \\
&= \mathcal{V}
\end{align*}
where the second last equality comes from Remark \ref{rem_prop_compactsupp}.
\end{proof}

\begin{thm}\label{thm_nonaffine_hyper_tele}
Let $X$ be a noetherian separated scheme with hypersurface singularities. Then the action of $D(X)$ on the singularity category $S(X)$ satisfies the relative telescope conjecture i.e., every smashing $D(X)$-submodule of $S(X)$ is generated by objects of $S(X)^c$. In particular, if $X$ is a hypersurface defined by a section of an ample line bundle on some ambient regular separated noetherian scheme $T$ as above then $S(X)$ satisfies the usual telescope conjecture.
\end{thm}
\begin{proof}
The result is an application of Theorem \ref{thm_rel_tele}. We have seen in Theorem \ref{thm_hyper_nonaffine_bijections} that $D(X)$-submodules are classified by $\Sing X$ via the assignments $\s$ and $\t$. By Proposition \ref{prop_nonaffine_specclosed} compact objects of $S(X)$ have specialization closed support. Finally, we have proved in the last lemma that every irreducible closed subset of $\Sing X$ can be realised as the support of a compact object.

Thus the conditions of Theorem \ref{thm_rel_tele} are met for the action of $D(X)$ on $S(X)$ and it follows that the relative telescope conjecture holds. In the case Corollary \ref{cor_ample_stable} applies this reduces to the usual telescope conjecture.
\end{proof}

\section{A general classification theorem}\label{sec_general}
We are now ready to prove a version of Theorem \ref{thm_hyper_nonaffine_bijections} valid in higher codimension. Our strategy is to reduce to the hypersurface case so we may deduce the result from what we have already proved. Let us begin by fixing some terminology and notation for the setup we will be considering following Section 2 of \cite{OrlovSing2}.

Throughout this section by a complete intersection ring we mean a ring $R$ such that there is a regular ring $Q$ and a surjection $Q\to R$ with kernel generated by a regular sequence. A locally complete intersection scheme $X$ is a closed subscheme of a regular scheme such that the corresponding sheaf of ideals is locally generated by a regular sequence. All schemes considered from this point onward are assumed to have enough locally free sheaves.
Let $T$ be a separated regular noetherian scheme of finite Krull dimension and $\mathcal{E}$ a vector bundle on $T$ of rank $c$. For a section $t\in H^0(T,\mathcal{E})$ we denote by $Z(t)$ the \emph{zero scheme} of $t$. We recall that $Z(t)$ can be defined globally by the exact sequence
\begin{displaymath}
\xymatrix{
\mathcal{E}^{\vee} \ar[r]^{t^{\vee}} & \str_T \ar[r] & \str_{Z(t)} \ar[r] & 0.
}
\end{displaymath}
It can also be defined locally by taking a cover $X= \Un_iU_i$ trivializing $\mathcal{E}$ via $f_i\colon \mathcal{E}\vert_{U_i} \stackrel{\sim}{\to} \str_{U_i}^{\oplus c}$ and defining an ideal sheaf $\mathscr{I}(s)$ by
\begin{displaymath}
\mathscr{I}(t)\vert_{U_i} = (f_i(t)_1,\ldots, f_i(t)_c).
\end{displaymath}
We say that the section $t$ is \emph{regular} if the ideal sheaf $\mathscr{I}(t)$ is locally generated by a regular sequence. Thus the zero scheme $Z(t)$ of a regular section $t$ is a locally complete intersection in $T$ of codimension $c$. In our situation $t$ is  regular if and only if $\codim Z(t) = \rk \mathcal{E} = c$ (cf.\ \cite{MatsuAlg} 16.B). 

Let $T$ and $\mathcal{E}$ be as above and let $t\in H^0(T,\mathcal{E})$ be a regular section with zero scheme $X$. Denote by $\mathcal{N}_{X/T}$ the normal bundle of $X$ in $T$. There are projective bundles $\mathbb{P}(\mathcal{E}^\vee) = T'$ and $\mathbb{P}(\mathcal{N}_{X/T}^\vee) = Z$ with projections which we denote $q$ and $p$ respectively. Associated to these projective bundles are canonical line bundles $\str_{\mathcal{E}}(1)$ and $\str_\mathcal{N}(1)$ together with canonical surjections
\begin{displaymath}
q^*\mathcal{E} \to \str_\mathcal{E}(1) \quad \text{and} \quad p^*\mathcal{N}_{X/S} \to \str_\mathcal{N}(1).
\end{displaymath}
The section $t$ induces a section $t'\in H^0(T', \str_\mathcal{E}(1))$ and we denote its divisor of zeroes by $Y$. The natural closed immersion $Z\to T'$ factors via $Y$. To summarize there is a commutative diagram
\begin{equation}\label{eq_setup}
\xymatrix{
Z = \mathbb{P}(\mathcal{N}_{X/T}^\vee) \ar[r]^(0.7)i \ar[d]_p & Y \ar[r]^(0.3)u \ar[dr]^{\pi} & \mathbb{P}(\mathcal{E}^\vee) = T' \ar[d]^{q} \\
X \ar[rr]_j && T.
}
\end{equation}
This gives rise to functors $Si_*\colon S(Z) \to S(Y)$ and $Sp^*\colon S(X) \to S(Z)$ by \cite{KrStab} Theorem 1.5 and Theorem 6.6 respectively. Orlov proves the following theorem in Section 2 of \cite{OrlovSing2}:

\begin{thm}\label{thm_Orlov}
Let $T,T',X,$ and $Y$ be as above. Then the functor
\begin{displaymath}
\Phi_Z:= i_*p^* \colon D^b(\Coh X) \to D^b(\Coh Y)
\end{displaymath}
induces an equivalence of triangulated categories
\begin{displaymath}
\overline{\Phi}_Z \colon D_{\mathrm{Sg}}(X) \to D_{\mathrm{Sg}}(Y).
\end{displaymath}
\end{thm}

\begin{rem}
In Orlov's paper it is assumed that all the schemes are over some fixed base field and this is used in the proof. However, this hypothesis turns out to be unnecessary. In Appendix A of \cite{BurkeMF2} Jesse Burke and Mark Walker give a proof of Orlov's theorem without this assumption; it is this version of the theorem that we state.
\end{rem}

We wish to show this equivalence extends to the infinite completions $S(X)$ and $S(Y)$; it is natural to ask if the theorem extends and considering the larger categories allows us to utilise the formalism we have developed. In order to show the equivalence extends we demonstrate that it is  compatible with the functor $Si_*Sp^*$, induced by $i$ and $p$ as in Section 6 of \cite{KrStab}, via $I_\l Q_\r$. General nonsense about triangulated categories then implies $Si_*Sp^*$ must also be an equivalence.

\begin{notation}
We will frequently be concerned below with commuting diagrams involving the functors of Theorem \ref{thm_recol} for pairs of schemes. As in \cite{KrStab} we will tend not to clutter the notation by indicating which scheme the various functors correspond to as it is always identifiable from the context.
\end{notation}

\begin{lem}\label{lem_equiv1}
Let $i\colon Z\to Y$ be a regular closed immersion i.e., the ideal sheaf on $Y$ defining $Z$ is locally generated by a regular sequence, where $Z$ and $Y$ are noetherian separated schemes. Then the functor 
\begin{displaymath}
\hat{R}i_*\colon K(\Inj Z)\to K(\Inj Y)
\end{displaymath}
of \cite{KrStab} Theorem 1.4 has a coproduct preserving right adjoint $K(i^!)$ and sends compact objects to compact objects.
\end{lem}
\begin{proof}
Since $i$ is a closed immersion we have an adjunction at the level of categories of quasi-coherent sheaves
\begin{displaymath}
\xymatrix{
\QCoh Z \ar[r]<0.5ex>^{i_*} \ar@{<-}[r]<-0.5ex>_{i^!} & \QCoh Y.
}
\end{displaymath}
The right adjoint $i^!$ sends injectives to injectives as $i_*$ is exact. 

These functors give an adjunction
\begin{displaymath}
\xymatrix{
K(\QCoh Z) \ar[r]<0.5ex>^{K(i_*)} \ar@{<-}[r]<-0.5ex>_{K(i^!)} & K(\QCoh Y)
}
\end{displaymath}
and $K(i^!)$ restricts to a functor from $K(\Inj Y)\to K(\Inj Z)$. We claim that this restricted functor is the right adjoint of $\hat{R}i_*$. Recall that $\hat{R}i_*$ is defined by the composite
\begin{displaymath}
\xymatrix{
K(\Inj Z) \ar[r]^(0.45){J} & K(\QCoh Z) \ar[r]^{K(i_*)} & K(\QCoh Y) \ar[r]^(0.55){J_\l} & K(\Inj Y)
}
\end{displaymath}
where $J$ is the inclusion and $J_\l$ is left adjoint to the corresponding inclusion for $Y$. For $A \in K(\Inj Z)$ and $B\in K(\Inj Y)$ there are isomorphisms
\begin{align*}
\Hom(\hat{R}i_*A,B) &= \Hom(J_\l K(i_*)JA,B) \\
&\iso \Hom(JA, K(i^!)JB) \\
&\iso \Hom(JA, JK(i^!)B) \\
&\iso \Hom(A, K(i^!)B)
\end{align*}
the first equality by definition, the third isomorphism $JK(i^!)\iso K(i^!)J$ as $K(i^!)$ sends complexes of injectives to complexes of injectives, and the fourth isomorphism as $J$ is fully faithful. This proves that the right adjoint to $\hat{R}i_*$ is induced by $K(i^!)$ as claimed.

To complete the proof note that $i^!$ preserves coproducts. The functor $K(i^!)$ and hence the right adjoint of $\hat{R}i_*$ are thus also coproduct preserving. It now follows from \cite{NeeGrot} Theorem 5.1 that $\hat{R}i_*$ sends compact objects to compact objects.
\end{proof}

Thus from \cite{KrStab}, namely the first diagram of Theorem 6.1 and Remark 3.8, we deduce, whenever $i$ is a regular closed immersion, a commutative square
\begin{equation}\label{eqblah}
\xymatrix{
D^b(\Coh Z) \ar[d]_{i_*} \ar[r]^{Q_\r}_{\sim} & K^c(\Inj Z) \ar[d]^{\hat{R}i_*} \\
D^b(\Coh Y) \ar[r]^{\sim}_{Q_\r} & K^c(\Inj Y).
}
\end{equation}


\begin{lem}\label{lem_equiv2}
Let $Z$ and $Y$ be Gorenstein separated schemes and suppose $i\colon Z\to Y$ is a regular closed immersion. Then the functor $K(i^!)$ sends acyclic complexes of injectives to acyclic complexes of injectives. 
\end{lem}
\begin{proof}
As $i$ is a regular closed immersion $i_*$ sends perfect complexes to perfect complexes. Thus $\mathbf{R}i^!\colon D(Y) \to D(Z)$ preserves coproducts by \cite{NeeGrot} Theorem 5.1 so is isomorphic to $\mathbf{L}i^*(-) \otimes \mathbf{R}i^!\str_Y$ by \emph{ibid}.\ Theorem 5.4. The scheme $Y$ is Gorenstein so $\mathbf{R}i^! \str_Y$ is a dualizing complex on $Z$. As $Z$ is also Gorenstein the dualizing complex $\mathbf{R}i^!\str_Y$ is (at least on each connected component) a suspension of an invertible sheaf. Thus we can choose $n\in \int$ so that $H^j(\mathbf{R}i^!F) = 0$ for every quasi-coherent sheaf $F$ on $Y$ and $j> n$ as $\mathbf{L}i^*(F)$ is always bounded above.

If $A$ is an acyclic complex of injectives then the truncation
\begin{displaymath}
0 \to A^0 \to A^1 \to A^2 \to \cdots
\end{displaymath} 
is an injective resolution of $B =\ker(A^0 \to A^1)$. Thus applying $K(i^!)$ to this truncation computes $\mathbf{R}i^!B$ so the resulting complex is acyclic above degree $n$. By taking suspensions we deduce that $K(i^!)A$ is in fact acyclic everywhere and we have already noted that $i^!$ preserves injectivity as it has an exact left adjoint.
\end{proof}

\begin{rem}\label{rem_regular}
As the notation in the last two lemmas indicates they apply to the situation we are interested in, namely the one given at the start of the section: the morphism $i\colon Z \to Y$ is a regular closed immersion. Let us indicate why this is the case. Pick some open affine subscheme $\Spec Q$ of $T$, with preimage in $X$ isomorphic to $\Spec R$, on which $\mathcal{E}$ is trivial and such that the kernel of $Q\to R$ is generated by the regular sequence $\{q_1,\ldots,q_c\}$. We get a diagram of open subschemes of the diagram (\ref{eq_setup})
\begin{displaymath}
\xymatrix{
\mathbb{P}^{c-1}_R \ar[r]^(0.7)i \ar[d]_p & Y' \ar[r]^(0.3)u \ar[dr]^{\pi} & \mathbb{P}^{c-1}_Q \ar[d]^{q} \\
\Spec R \ar[rr]_j && \Spec Q.
}
\end{displaymath}
The hypersurface $Y'$ is defined by the section $t' = \sum_{i=1}^c q_ix_i$ of $\str_{\mathbb{P}^{c-1}_Q}(1)$, where the $x_i$ are a basis for the global sections of $\str_{\mathbb{P}^{c-1}_Q}$. Let $z$ be a point in the $c$th standard open affine $\mathbb{A}^{c-1}_R$ in $\mathbb{P}^{c-1}_R$ (we choose this open affine for ease of notation, little changes if $z$ lies in another standard open affine) and consider the local maps of local rings
\begin{displaymath}
\str_{T',ui(z)} \stackrel{\a}{\to} \str_{Y,i(z)} \stackrel{\b}{\to} \str_{Z,z}.
\end{displaymath}
We wish to show that $\ker \b$ is generated by a regular sequence. Note that both $\a$ and $\b\a$ have kernels generated by regular sequences: the kernel of $\a$ is generated by the image of $s = q_1x_1+\cdots+q_{c-1}x_{c-1} + q_c$ in $\str_{T',ui(z)}$ and the kernel of $\b\a$ is generated by the image of the regular sequence $\{q_1,\ldots,q_c\}$.

It is clear that the image of $\{q_1,\ldots,q_{c-1},s\}$ is a regular sequence in $\str_{T',ui(z)}$ and as this ring is local and noetherian we may permute the order of the elements in this sequence and it remains regular by \cite{MatsuRing} Theorem 16.3. Thus $\{s,q_1,\ldots,q_{c-1}\}$ is a regular sequence in $\str_{T',ui(z)}$. It follows that the image of $\{q_1,\ldots,q_{c-1}\}$ is a regular sequence in $\str_{Y,i(z)}$ and it generates the kernel of $\b$. Thus $i$ is a regular closed immersion as claimed.
\end{rem}

So we have an adjoint pair of functors
\begin{displaymath}
\xymatrix{
K(\Inj Z) \ar[r]<0.5ex>^{\hat{R}i_*} \ar@{<-}[r]<-0.5ex>_{K(i^!)} & K(\Inj Y)
}
\end{displaymath}
which both send acyclic complexes to acyclic complexes: $\hat{R}i_*$ by Theorem 1.5 of \cite{KrStab} and $K(i^!)$ by Lemma \ref{lem_equiv2}. Thus they restrict to an adjoint pair
\begin{displaymath}
\xymatrix{
S(Z) \ar[r]<0.5ex>^{Si_*} \ar@{<-}[r]<-0.5ex>_{Si^!} & S(Y).
}
\end{displaymath}

So we have a commutative square
\begin{displaymath}
\xymatrix{
S(Y) \ar[r]^(0.4)I \ar[d]_{Si^!} & K(\Inj Y) \ar[d]^{K(i^!)} \\
S(Z) \ar[r]_(0.4)I & K(\Inj Z).
}
\end{displaymath}
Taking left adjoints of the functors in this last square we get another commutative diagram
\begin{equation*}\label{eqblah2}
\xymatrix{
K(\Inj Z) \ar[r]^(0.6){I_\l} \ar[d]_{\hat{R}i_*} & S(Z) \ar[d]^{Si_*} \\
K(\Inj Y) \ar[r]_(0.6){I_\l} & S(Y).
}
\end{equation*}
By Lemma \ref{lem_equiv1} the composite $I_\l \hat{R}i_*$ sends compact objects to compact objects. As $I_\l$ sends compacts to compacts and is essentially surjective, up to summands, on compacts we see that $Si_*$ must preserve compacts too. So restricting this square to compact objects and juxtaposing with the square (\ref{eqblah}) we get a commutative diagram
\begin{displaymath}
\xymatrix{
D^b(\Coh Z) \ar[d]_{i_*} \ar[r]^{Q_\r}_{\sim} & K^c(\Inj Z) \ar[d]^{\hat{R}i_*} \ar[r]^(0.6){I_\l} & S^c(Z) \ar[d]^{Si_*} \\
D^b(\Coh Z) \ar[r]^{\sim}_{Q_\r} & K^c(\Inj Y) \ar[r]_(0.6){I_\l} & S^c(Y).
}
\end{displaymath}
In particular, the functor $\overline{i}_*\colon D_{\mathrm{Sg}}(Z) \to D_\mathrm{Sg}(Y)$ induced by $i$ is compatible with $Si_*$ under the embeddings of $D_\mathrm{Sg}(Z)$ and $D_\mathrm{Sg}(Y)$ as the compact objects in $S(Z)$ and $S(Y)$.

\begin{prop}\label{cor_equiv}
There is an equivalence of triangulated categories
\begin{displaymath}
Si_*Sp^*\colon S(X) \to S(Y)
\end{displaymath}
which when restricted to compact objects is Orlov's equivalence.
\end{prop}
\begin{proof}
We have just seen that the square
\begin{displaymath}
\xymatrix{
D_\mathrm{Sg}(Z) \ar[r] \ar[d]_{\overline{i}_*} & S(Z) \ar[d]^{Si_*} \\
D_\mathrm{Sg}(Y) \ar[r] & S(Y)
}
\end{displaymath}
commutes. By \cite{KrStab} Theorem 6.6 the square
\begin{displaymath}
\xymatrix{
D_\mathrm{Sg}(X) \ar[r] \ar[d]_{\overline{p}^*} & S(X) \ar[d]^{Sp^*} \\
D_\mathrm{Sg}(Z) \ar[r] & S(Z)
}
\end{displaymath}
commutes. Putting this second square on top of the first the equivalence $\overline{\Phi}_Z$ fits into a commutative diagram
\begin{displaymath}
\xymatrix{
D_\mathrm{Sg}(X) \ar[r] \ar[d]_{\overline{\Phi}_Z}^{\wr} & S(X) \ar[d]^{Si_*Sp^*} \\
D_\mathrm{Sg}(Y) \ar[r] & S(Y).
}
\end{displaymath}
Hence $Si_*Sp^*$ is a coproduct preserving exact functor between compactly generated triangulated categories inducing an equivalence on compact objects. It follows from abstract nonsense that it must be an equivalence.
\end{proof}

We have thus reduced the problem of understanding $S(X)$ to that of understanding $S(Y)$. The scheme $Y$ is locally a hypersurface as it is a locally complete intersection in the regular scheme $T'$ and has codimension 1. Theorem \ref{thm_hyper_nonaffine_bijections} thus applies and we have the following theorem, where we use the notation introduced at the beginning of the section.

\begin{thm}\label{thm_general_bijections}
The category $D(Y)$ acts on $S(X)$ via the equivalence\\ $S(X) \iso S(Y)$ giving order preserving bijections
\begin{displaymath}
\left\{ \begin{array}{c}
\text{subsets of}\; \Sing Y 
\end{array} \right\}
\xymatrix{ \ar[r]<1ex>^{\tau} \ar@{<-}[r]<-1ex>_{\sigma} &} \left\{
\begin{array}{c}
\text{localizing} \; D(Y)\text{-submodules of} \; S(X) \\
\end{array} \right\} 
\end{displaymath}
and
\begin{displaymath}
\left\{ \begin{array}{c}
\text{specialization closed} \\ \text{subsets of}\; \Sing Y 
\end{array} \right\}
\xymatrix{ \ar[r]<1ex>^{\tau} \ar@{<-}[r]<-1ex>_{\sigma} &} \left\{
\begin{array}{c}
\text{localizing} \; D(Y)\text{-submodules} \\ \text{of} \; S(X) \; \text{generated by} \\ \text{objects of} \; S(X)^c
\end{array} \right\}. 
\end{displaymath}
Furthermore if the line bundle $\str_\mathcal{E}(1)$ is ample, for example if $S$ is affine, then every localizing subcategory of $S(X)$ is a $D(Y)$-submodule so one obtains a complete classification of the localizing subcategories of $S(X)$.
\end{thm}
\begin{proof}
Let us denote the equivalence $S(X)\stackrel{\sim}{\to}S(Y)$ by $\Psi$. We define an action of $D(Y)$ on $S(X)$ by setting, for $E\in D(Y)$ and $A\in S(X)$
\begin{displaymath}
E\square A = \Psi^{-1}(E\odot \Psi A).
\end{displaymath}
It is easily checked that this is in fact an action.

The equivalence $\Psi$ sends localizing subcategories (generated by objects of $S(X)^c$) to localizing subcategories (generated by objects of $S(Y)^c$). A localizing subcategory $\mathcal{L}\cie S(X)$ is a $D(Y)$-submodule if and only if for every $E\in D(Y)$
\begin{displaymath}
E\square \mathcal{L} = \Psi^{-1}(E\odot \Psi\mathcal{L}) \cie \mathcal{L}
\end{displaymath}
if and only if $E\odot \Psi \mathcal{L} \cie \Psi \mathcal{L}$. In other words $\mathcal{L}$ is a $D(Y)$-submodule if and only if $\Psi\mathcal{L}$ is a $D(Y)$-submodule. Thus the theorem follows from Theorem \ref{thm_hyper_nonaffine_bijections} as $Y$ is locally a hypersurface.

The last statement is a consequence of Corollary \ref{cor_ample_stable}.
\end{proof}

\begin{cor}\label{cor_general_tele}
The relative telescope conjecture holds for the action of $D(Y)$ on $S(X)$. In particular if $\str_\mathcal{E}(1)$ is ample then the usual telescope conjecture holds for $S(X)$.
\end{cor}
\begin{proof}
This is immediate from the corresponding statements for the action of $D(Y)$ on $S(Y)$ given in Theorem \ref{thm_nonaffine_hyper_tele}.
\end{proof}

\section{Embedding independence}\label{sec_indep}
To prove Theorem \ref{thm_general_bijections} we have relied on the choice of some ambient scheme $T$, vector bundle $\mathcal{E}$, and a regular section $t$ of $\mathcal{E}$. Thus it is not clear that the support theory one produces, via the hypersurface $Y$ associated to this data, is independent of the choices we have made. We now show this is in fact the case: the choices one makes do not matter as far as the support theory is concerned.

The setup will be exactly the same as previously, except we will have two possibly different regular noetherian separated schemes of finite Krull dimension $T_1$ and $T_2$ each carrying a vector bundle $\mathcal{E}_i$ with a regular section $t_i$ for $i=1,2$ such that
\begin{displaymath}
Z(t_1) \iso X \iso Z(t_2).
\end{displaymath}

Thus there are, by Proposition \ref{cor_equiv}, two equivalences
\begin{displaymath}
\Psi_1\colon S(X) \to S(Y_1) \quad \text{and} \quad \Psi_2 \colon S(X) \to S(Y_2)
\end{displaymath}
giving rise to a third equivalence $S(Y_1) \stackrel{\sim}{\to} S(Y_2)$ which we shall denote by $\Theta$.

We first treat the case in which both $\str_{\mathcal{E}_1}(1)$ and $\str_{\mathcal{E}_2}(1)$ are ample.

\begin{lem}\label{lem_ample_independence}
Suppose $\str_{\mathcal{E}_i}(1)$ is ample for $i=1,2$. Then there is a homeomorphism
\begin{displaymath}
\theta\colon \Sing Y_1 \to \Sing Y_2
\end{displaymath}
such that for any $A$ in $S(Y_1)$ we have
\begin{displaymath}
\theta \supp A = \supp \Theta A.
\end{displaymath}
In particular the two support theories for $S(X)$ obtained via the actions of $D(Y_1)$ and $D(Y_2)$ coincide up to this homeomorphism.
\end{lem}
\begin{proof}
We first define $\theta$ and show it is a bijection. Let $y$ be a point of $\Sing Y_1$. By Theorem \ref{thm_general_bijections} the subcategory $\mathit{\Gamma}_y S(Y_1)$ is a minimal localizing subcategory. Thus its essential image $\Theta \mathit{\Gamma}_y S(Y_1)$ is a minimal localizing subcategory of $S(Y_2)$. So by Corollary \ref{cor_ample_stable} the subcategory $\Theta \mathit{\Gamma}_y S(Y_1)$ is necessarily of the form $\mathit{\Gamma}_{\theta(y)}S(Y_2)$. This defines a function $\theta\colon \Sing Y_1 \to \Sing Y_2$ which is a bijection as $\Theta$ is an equivalence.

Let us now show that $\theta$ is compatible with supports. If $A$ is an object of $S(Y_1)$ then by Corollary \ref{cor_ample_stable} and Theorem \ref{thm_general_ltg} we have
\begin{displaymath}
\langle A \rangle_\mathrm{loc} = \langle \mathit{\Gamma}_yS(Y_1) \; \vert \; y\in \supp A \rangle_\mathrm{loc}.
\end{displaymath}
Applying $\Theta$ gives two sets of equalities, namely
\begin{align*}
\Theta \langle A \rangle_\mathrm{loc} = \langle \Theta A \rangle_\mathrm{loc} = \langle \mathit{\Gamma}_wS(Y_2) \; \vert \; w\in \supp \Theta A\rangle_\mathrm{loc}
\end{align*}
and
\begin{align*}
\Theta \langle A \rangle_\mathrm{loc} &= \Theta\langle \mathit{\Gamma}_yS(Y_1) \; \vert \; y\in \supp A\rangle_\mathrm{loc} \\
&= \langle \mathit{\Gamma}_{\theta(y)}S(Y_2) \; \vert \; y\in \supp A\rangle_\mathrm{loc}.
\end{align*}
We thus obtain $\theta \supp A = \supp \Theta A$ which shows that $\theta$ respects the support.

Finally, let us show that $\theta$ is a homeomorphism. Let $\mathcal{V}$ be a closed subset of $\Sing Y_1$. Then it follows from Lemma \ref{lem_closed_compact_exists} that there exists a compact object $c$ in $S(Y_1)$ whose support is $\mathcal{V}$. Hence
\begin{displaymath}
\theta \mathcal{V} = \theta \supp c = \supp \Theta c
\end{displaymath}
is closed by Proposition \ref{prop_nonaffine_specclosed} as $\Theta$ is an equivalence and so preserves compactness. The whole argument works just as well reversing the roles of $Y_1$ and $Y_2$ so $\theta^{-1}$ is also closed and thus $\theta$ is a homeomorphism.
\end{proof}

By working locally we are now able to extend this to arbitrary $X$ admitting a suitable embedding.

\begin{prop}\label{prop_indep}
Suppose we are given regular noetherian separated schemes of finite Krull dimension $T_1$ and $T_2$ each carrying a vector bundle $\mathcal{E}_i$ with a regular section $t_i$ for $i=1,2$ such that
\begin{displaymath}
Z(t_1) \iso X \iso Z(t_2).
\end{displaymath}
Then there is a homeomorphism $\theta\colon \Sing Y_1 \to \Sing Y_2$ satisfying
\begin{displaymath}
\theta \supp A = \supp \Theta A
\end{displaymath}
for any $A$ in $S(Y_1)$. In particular the two support theories for $S(X)$ obtained via the actions of $D(Y_1)$ and $D(Y_2)$ coincide up to this homeomorphism.
\end{prop}
\begin{proof}
Let $\{W_1^j\}_{j=1}^n$ and $\{W_2^k\}_{k=1}^m$ be open affine covers of $T_1$ and $T_2$. Denote by $\{U_1^j\}_{j=1}^n$ and $\{U_2^k\}_{k=1}^m$ the two open affine covers of $X$ obtained by restriction. For any of the open affines $W_i^l$ we may consider $\mathcal{E}_i\vert_{W_i^l}$ and $t_i\vert_{W_i^l}$; the zero scheme of $t_i\vert_{W_i^l}$ is precisely the open subscheme $U_i^l$ so each of the opens in the two covers satisfies the set up for Proposition \ref{cor_equiv} to apply. We denote by $Y_i^l$ the associated hypersurface. Furthermore, as $W_i^l$ is affine the canonical line bundle on $\mathbb{P}(\mathcal{E}_i\vert_{W_i^l}^\vee)$ is ample so Lemma \ref{lem_ample_independence} applies.

Now fix one of the $U_1^j \cie X$ and cover it by the open affines $U_{12}^{jk} = U_1^j \intersec U_2^k$ for $k=1,\ldots,m$. There are diagrams
\begin{displaymath}
\xymatrix{
& S(U_1^j) \ar[r]^{\Psi^j_1}_\sim & S(Y^j_1) \\
S(U_{12}^{jk}) \ar[ur] \ar[dr] & & \\
& S(U_2^k) \ar[r]^{\Psi^k_2}_\sim & S(Y^k_2)
}
\end{displaymath}
where the equivalences are the restrictions of $\Psi_1$ and $\Psi_2$ and the diagonal maps are inclusions. We thus get an equivalence
\begin{displaymath}
\Theta^{jk}\colon \Psi^j_1S(U_{12}^{jk}) \to \Psi_2^kS(U_{12}^{jk})
\end{displaymath}
restricting $\Theta$, and so as in Lemma \ref{lem_ample_independence} a support preserving homeomorphism
\begin{displaymath}
\theta^{jk} \colon \Sing Y_1^{jk} \to \Sing Y_2^{jk}
\end{displaymath}
where $Y_1^{jk}$ is the subset corresponding to $\Psi^j_1S(U_{12}^{jk})$ and $Y_2^{jk}$ corresponds to $\Psi^k_2S(U_{12}^{jk})$.

We have produced support preserving homeomorphisms $\theta^{jk}$ for each \\ $j = 1,\ldots,n$ and $k = 1,\ldots,m$ and the $Y_i^{jk}$ cover the singular locus of $Y_i$ for $i=1,2$. It just remains to note that these glue to the desired homeomorphism $\Sing Y_1 \to \Sing Y_2$; the required compatibility on overlaps is immediate as the $\theta^{jk}$ are defined via restrictions of $\Theta$.
\end{proof}

\section{Local complete intersections}\label{sec_lci}

Let us now restrict our attention to the case of local complete intersection rings. Theorem \ref{thm_general_bijections} applies in this case and we will explicitly describe the singular locus of the associated hypersurface $Y$; this can be done working with any choice of embedding as the associated support theory is invariant by the last subsection.

Suppose $(R,\mathfrak{m},k)$ is a local complete intersection of codimension $c$ i.e., $R$ is the quotient of a regular local ring $Q$ by an ideal generated by a regular sequence and
\begin{displaymath}
\cx_R k = \dim_k \mathfrak{m}/\mathfrak{m}^2 - \dim R = c,
\end{displaymath}
where $\cx_R$ is the complexity as defined in Section \ref{ssec_cx}. Replacing $Q$ by a quotient if necessary we may assume that the kernel of $Q\to R$ is generated by a regular sequence of length precisely $c$. To see this is the case suppose the kernel is generated by a regular sequence $\{q_1,\ldots, q_r\}$ with $r>c$. Then by considering the effect on the embedding dimension and the dimension of successive quotients by the $q_i$ we see that $r-c$ of the $q_i$ must lie in $\mathfrak{n}\setminus \mathfrak{n}^2$ where $\mathfrak{n}$ is the maximal ideal of $Q$. By \cite{MatsuRing} Theorem 16.3 any permutation of the $q_i$ is again a regular sequence so we may rearrange to first take the quotient by the $r-c$ of the $q_i$ not in $\mathfrak{n}^2$. This quotient is again regular, surjects onto $R$ and this surjection has kernel generated by a regular sequence of length $c$.

Set $X = \Spec R$, $T = \Spec Q$, $\mathcal{E} = \str_{T}^{\oplus c}$, and $t = (q_1,\ldots,q_c)$ where the $q_i$ are a regular sequence generating the kernel of $Q\to R$. Let $Y$ be the hypersurface defined by the section $\Sigma_{i=1}^c q_i x_i$ of $\str_{\mathbb{P}^{c-1}_Q}(1)$ where the $x_i$ are a basis for the free $Q$-module $H^0(\mathbb{P}^{c-1}_Q, \str_{\mathbb{P}^{c-1}_Q}(1))$. In summary we are concerned with the following commutative diagram
\begin{displaymath}
\xymatrix{
\mathbb{P}^{c-1}_R \ar[r]^(0.7)i \ar[d]_p & Y \ar[r]^(0.3)u \ar[dr]^{\pi} & \mathbb{P}^{c-1}_Q \ar[d]^{q} \\
X \ar[rr]_j && T.
}
\end{displaymath}


Let us first make the following trivial observation about the singular locus of $\mathbb{P}^{c-1}_R$.

\begin{lem}\label{lem_sing_projR}
There is an equality
\begin{displaymath}
\Sing \mathbb{P}^{c-1}_R = p^{-1} \Sing R.
\end{displaymath}
\end{lem}

Now we show that the singular locus of $Y$ can not be any bigger than the singular locus of $\mathbb{P}^{c-1}_R$.

\begin{lem}\label{lem_sing_Y}
The singular locus of $Y$, $\Sing Y$, is contained in $i(\Sing \mathbb{P}^{c-1}_R)$.
\end{lem}
\begin{proof}
We first show the singular locus of $Y$ is contained in the image of $i$. The image of $i$ is precisely $Y\intersec q^{-1}X$, so we want to show that away from $q^{-1}X$ the scheme $Y$ is regular. Let $\mathfrak{p}\in T \setminus X$, so the section $t = (q_1,\ldots, q_c)$ is not zero at $k(\mathfrak{p})$.  Thus in a neighbourhood of any point of $q^{-1}(\mathfrak{p})$ the section defining $Y\intersec q^{-1}(\mathfrak{p})$ is just a linear polynomial with invertible coefficients and so $Y$ is regular along its intersection with $q^{-1}(\mathfrak{p})$. Thus $\Sing Y \cie i(\mathbb{P}^{c-1}_R)$ as claimed.




Next let us prove that $\Sing Y$ is in fact contained in $i(\Sing\mathbb{P}^{c-1}_R)$. Given $z\in \mathbb{P}^{c-1}_R$ such that $i(z)\in \Sing Y$ we need to show $z\in \Sing \mathbb{P}^{c-1}_R$. By Remark \ref{rem_regular} the surjection
\begin{displaymath}
\str_{Y,i(z)} \to \str_{\mathbb{P}^{c-1}_R,z}
\end{displaymath}
has kernel generated by a regular sequence. From \cite{AGP} Proposition 5.2 we get inequalities
\begin{displaymath}
\cx_{\;\str_{\mathbb{P}^{c-1}_R,z}} k(z) \geq \cx_{\;\str_{Y,i(z)}} i_*k(z) = \cx_{\;\str_{Y,i(z)}}k(i(z)) > 0
\end{displaymath}
where we have also used \cite{AvramovExtremal} Theorem 3, so $z\in \Sing \mathbb{P}^{c-1}_R$.
\end{proof}

In fact the part of the singular locus of $Y$ corresponding to $\mathfrak{m}$ can not be any smaller than $p^{-1}(\mathfrak{m})$ either.

\begin{lem}\label{lem_sing_m}
Every point in $ip^{-1}(\mathfrak{m})$ is contained in $\Sing Y$.
\end{lem}
\begin{proof}
By Lemma \ref{lem_sing_projR} every point in $p^{-1}(\mathfrak{m})$ is singular in $\mathbb{P}^{c-1}_R$. Consider for $z\in p^{-1}(\mathfrak{m})$ the local maps
\begin{displaymath}
\str_{\mathbb{P}^{c-1}_Q,ui(z)} \stackrel{\a}{\to} \str_{Y,i(z)} \stackrel{\b}{\to} \str_{\mathbb{P}^{c-1}_R,z}
\end{displaymath} 
where the kernel of each of these morphisms and the composite is generated by a regular sequence (see Remark \ref{rem_regular}). We have assumed $Q\to R$ minimal i.e., the elements $q_i$ occuring in the regular sequence generating the kernel all lie in $\mathfrak{n}^2$ where $\mathfrak{n}$ is the maximal ideal of $Q$. Thus as $z$ lies over $\mathfrak{m}$ the image of each $q_i$ is in the square of the maximal ideal of $\str_{\mathbb{P}^{c-1}_Q,ui(z)}$.

By passing to a standard open affine in $\mathbb{P}^{c-1}_Q$ containing $ui(z)$ (and reordering the $q_i$ if necessary) we see that the morphism $\a$ has kernel generated by the image of $\sum_{i=1}^{c-1}q_ix_i + q_c$ where the $x_i$ are now coordinates on $\mathbb{A}^{c-1}_Q$. As the image of each $q_i$ is in the square of the maximal ideal of $\str_{\mathbb{P}^{c-1}_Q,ui(z)}$ the element $\sum_{i=1}^{c-1}q_ix_i + q_c$ defining $\str_{Y,i(z)}$ must also lie in the square of the maximal ideal. Hence $i(z)$ lies in $\Sing Y$.
\end{proof}

It follows from this that $\supp_{(D(Y),\square)} \mathit{\Gamma}_\mathfrak{m}S(R) = \mathbb{P}^{c-1}_k$. By Lemma \ref{lem_gen2} the object $I_\l Q_\r k$ generates $\mathit{\Gamma}_\mathfrak{m}S(R)$. Thus its image under $\overline{i}_*\overline{p}^*$, which is precisely $I_\l Q_\r$ of the structure sheaf of $ip^{-1}(\mathfrak{m})$ with the reduced induced scheme structure, generates $Si_*Sp^*\mathit{\Gamma}_\mathfrak{m}S(R)$. By Lemma \ref{lem_closed_compact_exists} this generating object has support, with respect to the $D(Y)$ action on $S(Y)$, precisely $ip^{-1}(\mathfrak{m})$. Thus, identifying the topological spaces $\mathbb{P}^{c-1}_k$ and $ip^{-1}(\mathfrak{m})$, we see $\mathit{\Gamma}_\mathfrak{m}S(R)$ has the claimed support.

We now show the singular locus of $Y$ is composed completely of such projective pieces with dimensions corresponding to the complexities of the residue fields of the points in $\Sing R$. 

\begin{prop}\label{prop_sing_computed}
As a set the singular locus of $Y$ is
\begin{displaymath}
\Sing Y \iso \coprod_{\mathfrak{p}\in \Sing R} \mathbb{P}^{c_\mathfrak{p} -1}_{k(\mathfrak{p})} 
\end{displaymath}
where $c_\mathfrak{p} = \cx_{R_\mathfrak{p}}k(\mathfrak{p})$ is the codimension of $R_\mathfrak{p}$. 
\end{prop}
\begin{proof}
We can write $\Sing Y$, using the classification of Theorem \ref{thm_general_bijections}, as
\begin{displaymath}
\Sing Y \iso \coprod_{\mathfrak{p} \in \Sing R} \supp_{(D(Y),\square)} \mathit{\Gamma}_\mathfrak{p}S(R).
\end{displaymath}
Again using the classification theorem and the independence results of the previous subsection we may compute the $D(Y)$-support of $\mathit{\Gamma}_\mathfrak{p}S(R) = \mathit{\Gamma}_\mathfrak{p}S(R_\mathfrak{p})$ over $R_\mathfrak{p}$. By the discussion before the proposition this is precisely $\mathbb{P}^{c_\mathfrak{p} -1}_{k(\mathfrak{p})}$. 
\end{proof}

This gives a refined version of Theorem \ref{thm_general_bijections} for local complete intersection rings. Before stating the result let us indicate an interesting special case. Suppose $E$ is an elementary abelian $p$-group and $k$ is a field of characteristic $p$. Then the group ring $kE$ is a local complete intersection and the following corollary essentially contains \cite{BIKstrat} Theorem 8.1 as a special case.

\begin{cor}\label{cor_ci_win}
Suppose $(R,\mathfrak{m},k)$ is a local complete intersection. Then there are order preserving bijections
\begin{displaymath}
\left\{ \begin{array}{c}
\text{subsets of}\; \\ \coprod_{\mathfrak{p}\in \Sing R}\limits \mathbb{P}^{c_\mathfrak{p} -1}_{k(\mathfrak{p})}
\end{array} \right\}
\xymatrix{ \ar[r]<1ex>^{\tau} \ar@{<-}[r]<-1ex>_{\sigma} &} \left\{
\begin{array}{c}
\text{localizing subcategories of} \; S(R) \\
\end{array} \right\} 
\end{displaymath}
and
\begin{displaymath}
\left\{ \begin{array}{c}
\text{specialization closed} \\ \text{subsets of}\; \Sing Y
\end{array} \right\}
\xymatrix{ \ar[r]<1ex>^{\tau} \ar@{<-}[r]<-1ex>_{\sigma} &} \left\{
\begin{array}{c}
\text{localizing subcategories} \\ \text{of} \; S(R) \; \text{generated by} \\ \text{objects of} \; S(R)^c
\end{array} \right\}.
\end{displaymath}
Furthermore the telescope conjecture holds for $S(R)$.
\end{cor}
\begin{proof}
We apply Theorem \ref{thm_general_bijections} setting $X = \Spec R$, $S = \Spec Q$, $\mathcal{E} = \str_{S}^{\oplus c}$, and $s = (q_1,\ldots,q_c)$ where the $q_i$ are a regular sequence generating the kernel of\\ $Q\to R$. The line bundle $\str_\mathcal{E}(1)$ is ample on $\mathbb{P}^{c-1}_Q$ so we obtain a complete classification of the localizing subcategories of $S(R)$ in terms of $\Sing Y$. By Proposition \ref{prop_sing_computed} the singular locus of $Y$ is, as a set, precisely the given disjoint union of projective spaces. The final statement is Corollary \ref{cor_general_tele}.
\end{proof}

\begin{rem}
A similar result has been announced by Iyengar \cite{IyengarLCI} for locally complete intersection rings essentially of finite type over a field.
\end{rem}

\begin{rem}
As in \cite{Buchweitzunpub} there is an equivalence $D_{\mathrm{Sg}}(R) \cong \underline{MCM}(R)$ where $\underline{MCM}(R)$ is the stable category of maximal Cohen-Macaulay modules over $R$. Thus the equivalence gives a classification of thick subcategories of the stable category of maximal Cohen-Macaulay modules.
\end{rem}

  \bibliography{greg_bib}

\end{document}